\documentclass[a4paper,12pt]{article}
\usepackage{latexsym,amsmath,amsthm,amssymb}
\usepackage{a4wide}
\usepackage{hyperref}
\usepackage{marginnote}
\usepackage{color}
\usepackage{cite}
\hypersetup{
pdftitle={Maximizers with constraints}   
pdfauthor={Van Hoang Nguyen},
colorlinks = true,
linkcolor = magenta,
citecolor = blue,
}

\theoremstyle{plain}
\newtheorem{theorem}{Theorem}[section]

\newtheorem{proposition}[theorem]{Proposition}
\newtheorem{lemma}[theorem]{Lemma}

\theoremstyle{definition}

\theoremstyle{remark}

\renewcommand{\thefootnote}{\arabic{footnote}}

\def\R{\mathbb R}

\def\al{\alpha}
\def\om{\omega}
\def\ga{\gamma}
\def\De{\Delta} 

\def\lam{\lambda}
\def\vphi{\varphi}
\def\ep{\epsilon}
\def\na{\nabla}
\def\lt{\left}
\def\rt{\right}

\numberwithin{equation}{section}

\title{Maximizers for the variational problems associated with Sobolev type inequalities under constraints}
\author{Van Hoang Nguyen\footnote{Institute of Research and Development, Duy Tan University, Da Nang, Vietnam,}
\footnote{Institut de Math\'ematiques de Toulouse, Universit\'e Paul Sabatier, 118 Route de Narbonne, 31062 Toulouse c\'edex 09, France.}
}

\begin{document}
\maketitle


\renewcommand{\thefootnote}{}

\footnote{Email: \href{mailto: Van Hoang Nguyen <van-hoang.nguyen@math.univ-toulouse.fr>}{van-hoang.nguyen@math.univ-toulouse.fr} and \href{mailto: Van Hoang Nguyen <vanhoang0610@yahoo.com>}{vanhoang0610@yahoo.com}.}

\footnote{2010 \emph{Mathematics Subject Classification\text}: 46E35, 26D10.}

\footnote{\emph{Key words and phrases\text}: Sobolev type embedding, variational problems, maximizers}

\renewcommand{\thefootnote}{\arabic{footnote}}
\setcounter{footnote}{0}

\begin{abstract}
We propose a new approach to study the existence and non-existence of maximizers for the variational problems associated with Sobolev type inequalities both in the subcritical case and critical case under the equivalent constraints. The method is based on an useful link between the attainability of the supremum in our variational problems and the attainablity of the supremum of some special functions defined on $(0,\infty)$. Our approach can be applied to the same problems related to the fractional Laplacian operators. Our main results are new in the critical case and in the setting of the fractional Laplacian operator which was left open in the work of Ishiwata and Wadade \cite{Ishiwata,Ishiwata1}. In the subcritical case, our approach provides a new and elementary proof of the results of Ishiwata and Wadade \cite{Ishiwata,Ishiwata1}.  
\end{abstract}

\section{Introduction}
Let $N \geq 2$ and $p \in (1,N]$. We denote by $p^*$ the critical exponent in the Sobolev embedding theorem, i.e., $p^* = Np/(N-p)$ if $p< N$ and $p^* =\infty$ if $p=N$. Let $\ga >0$ and $u \in W^{1,p}(\R^N)$, we denote
\[
\|u\|_{W^{1,p}_\ga} = \lt(\|\na u\|_p^\ga + \| u\|_p^\ga\rt)^{\frac1\ga},
\]
and $S_{N,p,\ga} = \{u \in W^{1,p}(\R^N): \|u\|_{W^{1,p}_\ga} =1\}$.

Given $\al >0$, let us consider the variational problem
\begin{equation}\label{eq:variationalproblem}
D_{N,p,\ga,\al,q} = \sup_{u\in S_{N,p,\ga}} \lt(\|u\|_p^p + \al \|u\|_q^q\rt) = \sup_{u \in S_{N,p,\ga}} J_{N,p,\ga,\al,q}(u),
\end{equation}
where $p < q \leq p^*$ if $p < N$, and $N < q < \infty$ if $p =N$. Such a problem was recently studied by Ishiwata and Wadade in the subcritical case $p \in (1,N]$ and $q \in (N,p^*)$ (see \cite{Ishiwata,Ishiwata1}). It has close relation with the well known Sobolev type inequality
\[
\|u\|_q^p \leq S (\|\na u\|_p^p + \|u\|_p^p)\qquad u\in W^{1,p}(\R^N)
\]
where $p \in (1,N]$ and $q \in (p,p^*]$ if $p < N$ and $q\in (N,\infty)$ if $p=N$, and $S$ is the best constant depending only on $N,p$ and $q$. The attainability of $S$ is well understood in literature. In the limit case $p =N$, we known that $W^{1,N}(\R^N) \not\hookrightarrow L^\infty(\R^N)$. In this case, the so-called Moser--Trudinger inequality is a perfect replacement. The Moser--Trudinger inequality in bounded domains was independently proved by Yudovi${\rm \check{c}}$ \cite{Y1961}, Poho${\rm \check{z}}$aev \cite{P1965} and Trudinger \cite{T1967}. It was then sharpened by Moser \cite{M1970} by finding the sharp exponent. The Moser--Trudinger inequality was generalized to unbounded domains by Adachi and Tanaka \cite{Adachi}, by Cao \cite{Cao} for the scaling invariant form in the subcritical case and by Ruf \cite{Ruf2005} and by Li and Ruf \cite{LiRuf2008} for the full $L^N$-norm form including the critical case. It takes the following form
\begin{equation}\label{eq:MTLiRuf}
\sup_{u\in W^{1,N}(\R^N),\, \int_{\R^N}(|\na u|^N + |u|^N) dx\leq 1} \int_{\R^N}\Phi_N(\al |u|^{\frac N{N-1}}) dx < \infty,
\end{equation}
for any $\al \leq \al_N:= N \om_{N-1}^{1/(N-1)}$, where $\om_{N-1}$ denotes the surface area of the unit sphere in $\R^N$, and $\Phi_N(t) = e^t -\sum_{k=0}^{N-2} \frac{t^k}{k!}$. We refer the interest reader to \cite{AD2004,AS07,AY10,doO2015,Nguyen2017} and references therein for the recent developments in the study of the Moser--Trudinger inequality. The existence of maximizer for the Moser--Trudinger inequality in bounded domains was proved by Carleson and Chang \cite{CC1986} for the unit ball, by Struwe \cite{Struwe} for domains near ball, by Flucher \cite{Flucher1992} for any domain in dimension two, and by Lin \cite{Lin1996} for any domain in $\R^N, N\geq 3$. The existence of maximizers for the Moser--Trudinger inequality in whole space $\R^N$ was proved by Ruf \cite{Ruf2005} in the dimension two with $\al =\al_2$, by Li and Ruf \cite{LiRuf2008} in any dimension $N$ with $\al =\al_N$, and by Ishiwata \cite{Ishi} for any $\al \in (0,\al_N)$ with $N \geq 3$ and $\al \in (2/B_2, \al_2]$ with $N=2$ where $B_2 = \sup\limits_{u\in W^{1,2}(\R^2) \setminus\{0\}}\frac{\|u\|_4^4}{\|u\|_2^2 \|\na u\|_2^2}$. Ishiwata also proved a non-existence result in dimension two when $\al >0 $ sufficiently small which is different with the bounded domain case. The proof
of Ishiwata relies on the careful analysis using the concentration--compactness type argument together with the behavior of the functional
\[
J_{N,\al}(u) = \|u\|_N^N + \frac{\al}{N} \|u\|_{\frac{N^2}{N-1}}^{\frac{N^2}{N-1}}
\]
which appears as two first terms of the $\int_{\R^N} \Phi_N(\al |u|^{\frac N{N-1}}) dx$. Note that $J_{N,\al}$ is exact our functional $J_{N,N,N,\al, \frac{N^2}{N-1}}$ above. The non-existence result even occurs in the higher dimension if we replace the condition $\|u\|_N^N + \|\na u\|_N^N =1$ by the condition $\|u\|_N^a + \|\na u\|_N^b =1$ with $a, b> 0$ (see \cite{doO16,LamLuZhang}). This non-existence result suggests that the attainability of the Moser--Trudinger inequality depends delicately on the choice of normalizing conditions even if conditions are equivalent. This leads to the study of the problem \ref{eq:variationalproblem}.


Concerning to the problem \eqref{eq:variationalproblem}, it was shown in \cite{Ishiwata,Ishiwata1} that the existence of maximizers for $D_{N,p,\ga,\al,q}$ is closed related to the exponent $\ga$. Especially, the following quantity
\begin{equation}\label{eq:auxiliaryprob}
\al_{N,p,q}(\ga) = \inf_{u\in S_{N,p,\ga}} \frac{1 -\|u\|_p^p}{\|u\|_q^q} = \inf_{u\in S_{N,p,\ga}} I_{N,p,\ga,q}(u),
\end{equation}
is the threshold for the attainability of $D_{N,p,\ga,\al,q}$. For $N \geq 2$, $p\in (1,N]$ and $q \in (p,p^*)$, let $B_{N,p,q}$ denote the best constant in the Gagliardo--Nirenberg--Sobolev inequality
\begin{equation}\label{eq:GNSinequality}
B_{N,p,q} =\sup_{u\in W^{1,p}(\R^N), \, u\not\equiv 0} \frac{\|u\|_q^q}{\|\na u\|_p^{\ga_c} \|u\|_p^{q -\ga_c}},\qquad \ga_c = \frac{N(q-p)}p.
\end{equation} 
A simple variational argument shows that $B_{N,p,q}$ is attained in $W^{1,p}(\R^N)$. Moreover, by a refined P\'olya--Szeg\"o principle due to Brothers and Ziemer \cite{Brothers}, all maximizers for $B_{N,p,q}$ is uniquely determined up to a translation, dilation and multiple by non-zero constant by a nonnegative, spherically symmetric and non-increasing function $v_* \in W^{1,p}(\R^N)$, i.e., if $u$ is maximizer for $B_{N,p,q}$ then $u(x) = c \lam^{1/q} v_*(\lam^{1/N}(x-x_0))$ a.e. in $\R^N$ for some $c\not=0$, $\lam >0$ and $x_0 \in \R^N$.

Among their results, Ishiwata and Wadade proved the following theorem on the effect of $\al$ and $\ga$ to the attainability of $D_{N,p,\ga,\al,q}$.
\begin{theorem}\label{IshiwataWadade}[Ishiwata and Wadade] 
Let $N \geq 2$, $p\in (1,N]$, $q \in (p,p^*)$, and $\al, \ga >0$. Then we have
\begin{description}
\item (1) If $\ga > \ga_c$ then $\al_{N,p,q}(\ga) =0$ and $D_{N,p,\ga,\al,q}$ is attained for any $\al >0$.
\item (2) If $\ga =\ga_c$ then $\al_{N,p,q}(\ga_c) = p/(B_{N,p,q} \ga_c)$ and $D_{N,p,\ga,\al,q}$ is attained for any $\al > \al_{N,p,q}(\ga_c)$ while it is not attained for any $\al \leq \al_{N,p,q}(\ga_c)$.
\item (3) If $\ga < \ga_c$ then $D_{N,p,\ga,\al,q}$ is attained for any $\al \geq \al_{N,p,q}(\ga)$ while it is not attained for any $\al < \al_{N,p,q}(\ga)$.
\end{description}
\end{theorem}
The proof of Theorem \ref{IshiwataWadade} given in \cite{Ishiwata,Ishiwata1} is mainly based on the rearrangement argument (which enables us to study the problem \eqref{eq:variationalproblem} only on radial functions) and the subcriticality of the problem \eqref{eq:variationalproblem}, i.e., $q < p^*$. In this case, the embedding $W^{1,p}_{\rm rad}(\R^N) \hookrightarrow L^q(\R^N)$ is compact, where $W^{1,p}_{\rm rad}(\R^N)$ denotes the subspace of $W^{1,p}(\R^N)$ of all radial functions. By a careful analysis of the functional $J_{N,p,\ga,\al,q}$ on \emph{normalizing vanishing sequences} ((NVS) for short) in $S_{N,p,\ga}$ (a sequence $\{u_n\}\subset S_{N,p,\ga}$ is a normalizing vanishing sequence if $u_n$ is radial function, $u_n \rightharpoonup 0$ weakly in $W^{1,p}(\R^N)$ and $\lim\limits_{R\to\infty}\lim\limits_{n\to\infty} \int_{B_R} \lt(|u_n|^p + \al |u_n|^q\rt) dx =0$), they proved that
\[
\sup\{\limsup_{n\to\infty} J_{N,p,\ga,\al,q}(u_n)\, :\, \{u_n\} \,\,\text{\rm is (NVS)}\} =1.
\]
Using this vanishing level, they can exclude the vanishing behavior of any maximizing sequence (which can be assumed to be spherically symmetric and non-increasing function) for $D_{N,p,\ga,\al,q}$ and hence obtain the existence and non-existence of maximizers for $D_{N,p,\ga,\al,q}$.

The critical case $q = p^*$ with $p \in (1,N)$ remains open from the works of Ishiwata and Wadade \cite{Ishiwata,Ishiwata1}. In this case, we will denote $D_{N,p,\ga,\al,p^*}, \al_{N,p,p^*}(\ga), J_{N,p,\ga,\al,p^*}$ and $I_{N,p,\ga,p^*}$ by $D_{N,p,\ga,\al}, \al_{N,p}(\ga), J_{N,p,\ga,\al}$ and $I_{N,p,\ga}$ respectively for sake of simple notation. Our goal of this paper is to treat this remaining open case. Let us mention here that studying the problem \eqref{eq:variationalproblem} in the critical case is more difficult and more complicate than studying the one in the subcritical case. The first difficult is that the embedding $W^{1,p}_{\rm rad}(\R^N) \hookrightarrow L^{p^*}(\R^N)$ is not compact which is different with the subcritical case where $W^{1,p}_{\rm rad}(\R^N) \hookrightarrow L^{q}(\R^N)$ with $q < p^*$ is compact. The second difficult is that the \emph{concentrating phenomena} of the maximizing sequence for $D_{N,p,\ga,\al}$ can occur (a sequence $\{u_n\}_n\subset S_{N,p,\ga}$ is called a \emph{normalizing concentrating sequence} ((NCS) for short) if $u_n \rightharpoonup 0$ weakly in $W^{1,p}(\R^N)$ and $\int_{B_R^c} (|\na u_n|^N + |u_n|^N) dx \to 0$ as $n\to \infty$ for any $R >0$). In order to prove the attainability of $D_{N,p,\ga,\al}$, the usual way is to exclude both the vanishing and concentrating behavior of the maximizing sequences for $D_{N,p,\ga,\al}$ by using the concentration--compactness principle due to Lions \cite{PL85I,PL85II} or by computing exactly the vanishing level defined as above and the concentrating level of $J_{N,p,\ga,\al}$, i.e., $\sup\{\limsup_{n\to\infty} J_{N,p,\ga,\al}(u_n)\, :\, \{u_n\} \,\,\text{\rm is (NCS)}\},$ then using Brezis--Lieb lemma \cite{BrezisLieb} to gain a compact result for this maximizing sequence. It seems that these approaches are not easy to solve the problem \eqref{eq:variationalproblem} in the critical case. In this paper, we give an elementary proof without using any rearrangement argument for the problem \eqref{eq:variationalproblem} in the critical case. Moreover, our approach also gives information on the maximizer for $D_{N,p,\ga,\al}$ if exist.

Instead of using the sharp Gagliardo--Nirenberg--Sobolev inequality \eqref{eq:GNSinequality}, the following sharp Sobolev inequality plays an important role in our analysis below
\begin{equation}\label{eq:Sobolevinequality}
\|u\|_{p^*} \leq S_{N,p} \|\na u\|_p\qquad\forall\, u\in W^{1,p}(\R^N),
\end{equation}
where $S_{N,p}$ is the sharp constant which depends only on $N$ and $p$. Noet that the Sobolev inequality \eqref{eq:Sobolevinequality} is exact the critical case of the Gagliardo--Nirenberg--Sobolev inequality \eqref{eq:GNSinequality}. The precise value of $S_{N,p}$ is well known (see, e.g., \cite{Aubin,Talenti,CNV}). We do not mention its precise value here since it is not important in our analysis below. It is also well known that $S_{N,p}$ is not attained in $W^{1,p}(\R^N)$ unless $p < N^{1/2}$. However, $S_{N,p}$ is attained in a larger space than $W^{1,p}(\R^N)$, that is, the homogeneous Sobolev space $\dot{W}^{1,p}(\R^N)$ which is completion of $C_0^\infty(\R^N)$ under the Dirichlet norm $\|\na u\|_p$ with $u\in C^{\infty}_0(\R^N)$. The Sobolev inequality \eqref{eq:Sobolevinequality} also holds true in $\dot{W}^{1,p}(\R^N)$ with the same constant $S_{N,p}$ and the equality occurs if $u(x) =a u_*(b (x-x_0))$ for some $a, b>0$, $x_0 \in \R^N$ with
\begin{equation}\label{eq:optimalSobolev}
u_*(x) = (1 + |x|^{\frac p{p-1}})^{-\frac{N-p}p}.
\end{equation}
It is clear that $u_* \not\in W^{1,p}(\R^N)$ unless $p < N^{1/2}$. We refer the interest reader to the papers \cite{Aubin,CNV,Talenti} for more about the sharp Sobolev inequality \eqref{eq:Sobolevinequality} and its extremal functions.

The main result (on the existence and non-existence of maximizers for the problem \eqref{eq:variationalproblem} in the critical case) of this paper is the following theorem.
\begin{theorem}\label{Maximizerscritical}
Let $N \geq 2$, $p\in (1,N)$ and $\al, \ga >0$. Then $D_{N,p,\ga,\al}$ is not attained if $p \geq N^{1/2}$. If $1 < p < N^{1/2}$, we then have
\begin{description}
\item (1) If $\ga > p^*$ then $\al_{N,p}(\ga) =0$ and $D_{N,p,\ga,\al}$ is attained for any $\al >0$.
\item (2) If $\ga =p^*$ then $\al_{N,p}(p^*) = p/(p^* S_{N,p}^{p^*})$ and $D_{N,p,p^*,\al}$ is attained for any $\al > \al_{N,p}(p^*)$ while it is not attained for any $\al \leq \al_{N,p}(p^*)$. Moreover, if $\al \leq \al_{N,p}(p^*)$ then $D_{N,p,\ga,\al} =1$.
\item (3) If $p < \ga < p^*$ then $D_{N,p,\ga,\al}$ is attained for any $\al \geq \al_{N,p}(\ga)$ while it is not attained for any $\al < \al_{N,p}(\ga)$. Moreover, if $\al \leq \al_{N,p}(\ga)$ then $D_{N,p,\ga,\al} =1$. 
\item (4) If $\ga \leq p$ then $\al_{N,p}(\ga) =S_{N,p}^{-p^*}$ and $D_{N,p,\ga,\al}$ is not attained for any $\al >0$. In this case, we have
\[
D_{N,p,\ga,\al} = \max\{1, \al S_{N,p}^{p^*}\}.
\]
\end{description}
\end{theorem}
Theorem \ref{Maximizerscritical} together with Theorem \ref{IshiwataWadade} of Ishiwata and Wadade completes the study of the variational problem \eqref{eq:variationalproblem}. The main key in the proof of Theorem \ref{Maximizerscritical} is the following expression of $D_{N,p,\ga,\al}$ and $\al_{N,p}(\ga)$.
\begin{theorem}\label{formulacritical}
Let $N \geq 2$, $p\in (1,N)$ and $\al, \ga >0$. Then it holds
\begin{equation}\label{eq:Dformulacritical}
D_{N,p,\gamma,\al} = \sup_{t >0}\, \frac{(1+ t)^{\frac{p^*-p}\ga} + \al S_{N,p}^{p^*} t^{\frac{p^*}\gamma}}{(1+ t)^{\frac{p^*}\ga}} =:\sup_{t >0}\, f_{N,p,\ga,\al}(t),
\end{equation}
and
\begin{equation}\label{eq:alphaformulacritical}
\alpha_{N,p}(\gamma) = \frac1{S_{N,p}^{p^*}} \inf_{t > 0} \, (1+t)^{\frac{p^*-p}\ga} \frac{(1+ t)^{\frac p\ga} -1}{t^{\frac{p^*}\ga}} =: \frac1{S_{N,p}^{p^*}} \inf_{t > 0} \, g_{N,p,\ga}(t).
\end{equation}
\end{theorem}
The equalities \eqref{eq:Dformulacritical} and \eqref{eq:alphaformulacritical} give us the simple and computable expressions of $D_{N,p,\ga,\al}$ and $\al_{N,p}(\ga)$. Their proofs are mainly based on the Sobolev inequality \eqref{eq:Sobolevinequality} and on the estimates on test function constructed from the function $u_*$ above. Let us explain the main ideas in the proof of Theorem \ref{Maximizerscritical}. The non-existence result in the case $p \geq N^{1/2}$ is simply a consequence of the non-attainability of the Sobolev inequality \eqref{eq:Sobolevinequality} in $W^{1,p}(\R^N)$. In the case $p < N^{1/2}$, the proof is mainly based on the remark that $D_{N,p,\ga,\al}$ is attained in $S_{N,p,\ga}$ if and only if the function $f_{N,p,\ga,\al}$ attains its maximum value in $(0,\infty)$ (see Lemma \ref{relationcritical} below). This remark enables us reduce studying the attainability of $D_{N,p,\ga,\al}$ to studying the attainability of $\sup_{t>0} f_{N,p,\ga,\al,p}(t)$ in $(0,\infty)$. The latter problem is more elementary. We can use some simple differential calculus to solve it. Moreover, we will see from our proof that, whenever $\sup_{t >0} f_{N,p,\ga,\al}$ is attained by some $t_0 \in (0,\infty)$, we can construct explicitly a maximizer for $D_{N,p,\ga,\al}$. Our constructed maximizer comes from the maximizer of the Sobolev inequality \eqref{eq:Sobolevinequality}. 

It is interesting that our approach to prove Theorem \ref{Maximizerscritical} also gives a new proof of Theorem \ref{IshiwataWadade}. Comparing with the proof of Ishiwata and Wadade, our proof is more simple and more elementary. As in the proof of Theorem \ref{Maximizerscritical}, we will establish the following simple and computable expression for both $D_{N,p,\ga,\al,q}$ and $\al_{N,p,q}(\ga)$.
\begin{theorem}\label{explicitformula}
Let $N \geq 2$, $p\in (1,N]$, $q \in (p,p^*)$ and $\al, \ga >0$. Then it holds
\begin{equation}\label{eq:Dformula}
D_{N,p,\ga,\al,q} = \sup_{t >0}\, \frac{(1+t)^{\frac{q-p}\ga} + \al B_{N,p,q} t^{\frac{\ga_c}\ga}}{(1+t)^{\frac q\ga}} =: \sup_{t>0}\, f_{N,p,\ga,\al,q}(t),
\end{equation}
and
\begin{equation}\label{eq:alphaformula}
\alpha_{N,p,q}(\ga) =\frac1{B_{N,p,q}}\, \inf_{t >0}\, \frac{(1+t)^{\frac{q-p}\ga} \lt((1+t)^{\frac p\ga} -1\rt)}{t^{\frac{\ga_c}\ga}} =: \inf_{t>0}\, g_{N,p,\ga,q}(t).
\end{equation}
\end{theorem}
The equalities \eqref{eq:Dformula} and \eqref{eq:alphaformula} also will be the key in the proof of Theorem \ref{IshiwataWadade}. Similar with the proof of Theorem \ref{Maximizerscritical}, we also have a remark that $D_{N,p,\ga,\al,q}$ is attained in $S_{N,p,\ga}$ if and only $\sup_{t >0} f_{N,p,\ga,\al,q}$ is attained in $(0,\infty)$ (see Lemma \ref{subrelation} below). By this remark, we only have to study the attainability of $\sup_{t >0} f_{N,p,\ga,\al,q}$ in $(0,\infty)$. We will see below that the study of $\sup_{t>0}f_{N,p,\ga,\al,q}(t)$ is easier than the one of $\sup_{t>0} f_{N,p,\ga,\al}(t)$. Indeed, the behavior of $f_{N,p,\ga,\al,q}$ at infinity does not play any role in the study of $\sup_{t>0}f_{N,p,\ga,\al,q}(t)$ due to the subcriticality of the problem (it is clear that $\lim_{t\to \infty} f_{N,p,\ga,\al,q}(t) =0$). Hence, it is enough to study the behavior of $f_{N,p,\ga,\al,q}$ when $t$ is near $0$. Contrary with the subcritical case, the behavior of $f_{N,p,\ga,\al}$ at infinity is important in our analysis in the critical case. In fact, this is corresponding to the appearance of the concentration phenomena in studying the maximizing sequence for $D_{N,p,\ga,\al}$. 

Another advantage of our approach is that it can be applied to the variational problems of type \eqref{eq:variationalproblem} related to the fractional Sobolev type inequality (i.e., for the fractional Laplacian operator $(-\De)^s$ with $s \in (0,N/2)$). It seems that the method of Ishiwata and Wadade can not be applied immediately to these problems due to the non-locality of the problems. The results in this setting are given in Theorem \ref{sub} and Theorem \ref{critical} below both in subcritical case and critical case. Since the proof of these theorems is completely similar with the one of Theorem \ref{IshiwataWadade} and Theorem \ref{Maximizerscritical}. So we only sketch their proof in section \S4 below.


The rest of this paper is constructed as follows. In the next section, we will prove Theorem \ref{Maximizerscritical} on the effect of $p, \al$ and $\ga$ on the existence and non-existence of maximizers for the variational problem \eqref{eq:variationalproblem} in the critical case. Theorem \ref{formulacritical} is also proved in this section. In section \S3, we prove Theorem \ref{explicitformula} and apply it to give a new proof of Theorem \ref{IshiwataWadade} following the approach to Theorem \ref{Maximizerscritical}. In the last section \S4, we explain how our approach can be applied to the similar problems for the fractional Laplacian opertors.



\section{Proof of Theorem \ref{Maximizerscritical}}
This section is devoted to prove Theorem \ref{Maximizerscritical}. We first give a proof of Theorem \ref{formulacritical} on the expressions of $D_{N,p,\ga,\al}$ and $\al_{N,p}(\ga)$ which will be used in our proof of Theorem \ref{Maximizerscritical}.
\begin{proof}[Proof of Theorem \ref{formulacritical}]
We first prove \eqref{eq:Dformulacritical}. For any $u \in S_{N,p,\ga}$, by Sobolev inequality \eqref{eq:Sobolevinequality} we get
\begin{align*}
J_{N,p,\al,\ga}(u) \leq \|u\|_p^p + \al S_{N,p}^{p^*} \|\na u\|_p^{p^*} &=\frac{\|u\|_p^p (\|u\|_p^\gamma + \|\na u\|_p^\ga)^{\frac{p^*-p}\ga} + \al S_{N,p}^{p^*} \|\na u\|_p^{p^*}}{(\|u\|_p^\ga + \|\na u\|_p^\ga)^{\frac{p^*}\ga}}\notag\\
&= \frac{\lt(1+ \frac{\|\na u\|_p^\ga}{\|u\|_p^\ga}\rt)^{\frac{p^*-p}\ga} + \al S_{N,p}^{p^*} \lt(\frac{\|\na u\|_p^\ga}{\|u\|_p^\ga}\rt)^{\frac{p^*}\ga}}{\lt(1+ \frac{\|\na u\|_p^\ga}{\|u\|_p^\ga}\rt)^{\frac{p^*}\ga}}\notag\\
&\leq \sup_{t >0}\, \frac{(1+ t)^{\frac{p^*-p}\ga} + \al S_{N,p}^{p^*} t^{\frac{p^*}\gamma}}{(1+ t)^{\frac{p^*}\ga}}.
\end{align*}
Taking the supremum over all $u\in S_{N,p,\ga}$, we get
\begin{equation}\label{eq:aaaa}
D_{N,p,\ga,\al} \leq \sup_{t >0}\, \frac{(1+ t)^{\frac{p^*-p}\ga} + \al S_{N,p}^{p^*} t^{\frac{p^*}\gamma}}{(1+ t)^{\frac{p^*}\ga}}.
\end{equation}

For a reversed inequality of \eqref{eq:aaaa}, we take $u_*$ to be a maximizer for the Sobolev inequality \eqref{eq:Sobolevinequality} (we can take $u_*$ as in \eqref{eq:optimalSobolev}). Our aim is to construct the test functions to estimate $D_{N,p,\ga,\al}$ from $u_*$. Note that $u_* \not\in W^{1,p}(\R^N)$ in general, hence we can not use directly it as a test function for $D_{N,p,\ga,\al}$. To proceed, we will truncate the function $u_*$ as follows. Let $\varphi \in C_0^\infty(\R^N)$ be a nonnegative function such that $0\leq \vphi \leq 1$, $\vphi(x) =1$ if $|x|\leq 1$ and $\vphi(x) =0$ if $|x| \geq 2$. For any $R >0$ define $u_{*,R}(x) = u_*(x) \vphi(x/R)$ then $u_{*,R} \in W^{1,p}(\R^N)$ and 
\[
\lim_{R\to \infty} \|u_{*,R}\|_{p^*}^{p^*} = \|u_*\|_{p^*}^{p^*}\quad\text{\rm and }\quad \lim_{R \to \infty} \|\na u_{*,R}\|_p^p = \|\na u_*\|_p^p.
\]
Consequently, we have
\[
\lim_{R\to \infty} \frac{\|u_{*,R}\|_{p^*}}{\|\na u_{*,R}\|_p} = S_{N,p}.
\]
For any $\ep >0$, we can choose $R >0$ such that $\|u_{*,R}\|_{p^*} \geq (S_{N,p}-\ep)\|\na u_{*,R}\|_p$. For any $\lam >0$ define
\begin{equation}\label{eq:testfunctioncritical}
w_{R,\lam}(x) = \frac{\lam^{\frac1p} u_{*,R}(\lam^{\frac1N} x)}{\lt(\|u_{*,R}\|_p^\gamma + \lam^{\frac \ga N} \|\na u_{*,R}\|_p^\ga\rt)^{\frac1\ga}}.
\end{equation}
We have $w_{R,\lam} \in S_{N,p,\ga}$, then it holds
\begin{align*}
D_{N,p,\ga,\al}&\geq J_{N,p,\ga,\al}(w_{R,\lam})=\frac{\|u_{*,R}\|_p^p}{\lt(\|u_{*,R}\|_p^\ga + \lam^{\frac\ga N} \|\na u_{*,R}\|_p^\ga\rt)^{\frac p\ga}} + \al \frac{\lam^{\frac{p^* -p}p} \|u_{*,R}\|_{p^*}^{p^*}}{\lt(\|u_{*,R}\|_p^\ga + \lam^{\frac\ga N} \|\na u_{*,R}\|_p^\ga\rt)^{\frac {p^*}\ga}}\notag\\
&\geq \frac{\|u_{*,R}\|_p^p}{\lt(\|u_{*,R}\|_p^\ga + \lam^{\frac\ga N} \|\na u_{*,R}\|_p^\ga\rt)^{\frac p\ga}} + \al (S_{N,p} -\ep)^{p^*} \frac{\lam^{\frac{p^* -p}p} \|\na u_{*,R}\|_{p}^{p^*}}{\lt(\|u_{*,R}\|_p^\ga + \lam^{\frac\ga N} \|\na u_{*,R}\|_p^\ga\rt)^{\frac {p^*}\ga}}\notag\\
&\geq \lt(1 -\frac{\ep}{S_{N,p}}\rt)^{p^*}\lt[\frac{\|u_{*,R}\|_p^p}{\lt(\|u_{*,R}\|_p^\ga + \lam^{\frac\ga N} \|\na u_{*,R}\|_p^\ga\rt)^{\frac p\ga}} +  \frac{\al S_{N,p}^{p^*}\,\lam^{\frac{p^* -p}p} \|\na u_{*,R}\|_{p}^{p^*}}{\lt(\|u_{*,R}\|_p^\ga + \lam^{\frac\ga N} \|\na u_{*,R}\|_p^\ga\rt)^{\frac {p^*}\ga}}\rt]\\
&=\lt(1 -\frac{\ep}{S_{N,p}}\rt)^{p^*} \frac{\lt(1+ \lam^{\frac\ga N} \frac{\|\na u_{*,R}\|_p^\ga}{\|u\|_p^\ga}\rt)^{\frac{p^*-p}\ga} + \al S_{N,p}^{p^*}\lt(\lam^{\frac \ga N} \frac{\|\na u_{*,R}\|_p^\ga}{\|u\|_p^\ga}\rt)^{\frac{p^*}\ga}}{\lt(1+ \lam^{\frac\ga N} \frac{\|\na u_{*,R}\|_p^\ga}{\|u\|_p^\ga}\rt)^{\frac{p^*}\ga}}.
\end{align*}
The preceding inequality holds for any $\lam >0$, hence
\[
D_{N,p,\ga,\al} \geq \lt(1 -\frac{\ep}{S_{N,p}}\rt)^{p^*}\sup_{t >0}\, \frac{(1+ t)^{\frac{p^*-p}\ga} + \al S_{N,p}^{p^*} t^{\frac{p^*}\gamma}}{(1+ t)^{\frac{p^*}\ga}},
\]
for any $\ep >0$. Letting $\ep \to 0^+$, we get
\begin{equation}\label{eq:bbbb}
S_{N,p,\ga,\al} \geq \sup_{t >0}\, \frac{(1+ t)^{\frac{p^*-p}\ga} + \al S_{N,p}^{p^*} t^{\frac{p^*}\gamma}}{(1+ t)^{\frac{p^*}\ga}}.
\end{equation}
Combining \eqref{eq:aaaa} and \eqref{eq:bbbb} together, we obtain \eqref{eq:Dformulacritical}.

We next prove \eqref{eq:alphaformulacritical}. For any function $u\in S_{N,p,\ga}$, using the sharp Sobolev inequality \eqref{eq:Sobolevinequality}, we get
\begin{align}\label{eq:usingSobolev}
I_{N,p,\ga}(u) \geq \frac1{S_{N,p}^{p^*}} \frac{1 -\|u\|_p^p}{\|\na u\|_p^{p^*}}&=\frac1{S_{N,p}^{p^*}}\frac{\lt(\|\na u\|_p^\ga + \|u\|_p^\ga\rt)^{\frac{p^*-p}\ga} \lt(\lt(\|\na u\|_p^\ga + \|u\|_p^\ga\rt)^{\frac p\ga} -\|u\|_p^p\rt)}{\|\na u\|_p^{p^*}}\notag\\
&=\frac1{S_{N,p}^{p^*}}\frac{\lt(\frac{\|\na u\|_p^\ga}{\|u\|_p^\ga} + 1\rt)^{\frac{p^*-p}\ga}\lt(\lt(\frac{\|\na u\|_p^\ga}{\|u\|_p^\ga}+ 1 \rt)^{\frac p\ga} -1\rt)}{\lt(\frac{\|\na u\|_p^\ga}{\|u\|_p^\ga}\rt)^{\frac {p^*}\ga}}\notag\\
&\geq \frac1{S_{N,p}^{p^*}} \inf_{t > 0} \, (1+t)^{\frac{p^*-p}\ga} \frac{(1+ t)^{\frac p\ga} -1}{t^{\frac{p^*}\ga}}.
\end{align}
We finish our proof by showing a reversed inequality of \eqref{eq:usingSobolev}. For any $\ep >0$, we choose $w_{R,\lam}$ defined by \eqref{eq:testfunctioncritical} as test function. We then have
\begin{align}\label{eq:Sob*}
\al_{N,p}(\gamma) &\leq I_{N,p,\gamma}(w_{R,\lam})\notag\\
& =\lt(1 - \frac{\|u_{*,R}\|_p^p}{\lt(\|u_{*,R}\|_p^\gamma + \lam^{\frac \ga N} \|\na u_{*,R}\|_p^\ga\rt)^{\frac p\ga}}\rt) \lt(\frac{\lam^{\frac{p^*-p}p} \|u_{*,R}\|_{p^*}^{p^*}}{\lt(\|u_{*,R}\|_p^\gamma + \lam^{\frac \ga N} \|\na u_{*,R}\|_p^\ga\rt)^{\frac {p^*}\ga}}\rt)^{-1} \notag\\
&\leq \frac1{(S_{N,p} -\ep)^{p^*}} \lt(1 - \frac{\|u_{*,R}\|_p^p}{\lt(\|u_{*,R}\|_p^\gamma + \lam^{\frac \ga N} \|\na u_{*,R}\|_p^\ga\rt)^{\frac p\ga}}\rt)\lt(\frac{\lam^{\frac{p^*-p}p} \|\na u_{*,R}\|_p^{p^*}}{\lt(\|u_{*,R}\|_p^\gamma + \lam^{\frac \ga N} \|\na u_{*,R}\|_p^\ga\rt)^{\frac {p^*}\ga}}\rt)^{-1}\notag\\
&= \frac1{(S_{N,p} -\ep)^{p^*}}\lt(1 +  \frac{\lam^{\frac \ga N}\|\na u_{*,R}\|_p^\ga}{\|u_{*,R}\|_p^\gamma}\rt)^{\frac {p^*-p}\ga}\frac{\lt(1 +  \frac{\lam^{\frac \ga N}\|\na u_{*,R}\|_p^\ga}{\|u_{*,R}\|_p^\gamma}\rt)^{\frac {p}\ga} -1}{\lt( \frac{\lam^{\frac \ga N}\|\na u_{*,R}\|_p^\ga}{\|u_{*,R}\|_p^\gamma}\rt)^{\frac{p^*}\ga}}.
\end{align}
The inequality \eqref{eq:Sob*} holds for any $\lam >0$, hence it holds
\[
\al^*(\gamma) \leq \frac1{(S_{N,p} -\ep)^{p^*}} \inf_{t > 0} \, (1+t)^{\frac{p^*-p}\ga} \frac{(1+ t)^{\frac p\ga} -1}{t^{\frac{p^*}\ga}},
\]
for any $\ep >0$. Letting $\ep \to 0$, we get 
\begin{equation}\label{eq:reverseineq}
\al^*(\ga) \leq \frac1{S_{N,p}^{p^*}} \inf_{t > 0} \, (1+t)^{\frac{p^*-p}\ga} \frac{(1+ t)^{\frac p\ga} -1}{t^{\frac{p^*}\ga}}.
\end{equation}
Combining \eqref{eq:usingSobolev} and \eqref{eq:reverseineq} we obtain \eqref{eq:alphaformulacritical}.
\end{proof}

Recall that
\[
f_{N,p,\ga,\al}(t) = \frac{(1+ t)^{\frac{p^*-p}\ga} + \al S_{N,p}^{p^*} t^{\frac{p^*}\gamma}}{(1+ t)^{\frac{p^*}\ga}},
\]
and
\[
g_{N,p,\ga}(t) = (1+t)^{\frac{p^*-p}\ga} \frac{(1+ t)^{\frac p\ga} -1}{t^{\frac{p^*}\ga}}.
\]

From the proof of \eqref{eq:aaaa}, we see that
\begin{equation}\label{eq:sosanh}
J_{N,p,\ga,\al}(u) \leq f_{N,p,\ga,\al}\lt(\frac{\|\na u\|_p^\ga}{\|u\|_p^\ga}\rt)\qquad \forall\, u\in S_{N,p,\ga}.
\end{equation}
Moreover, if $p < N^{1/2}$ then $u_* \in W^{1,p}(\R^N)$ and $\|u_*\|_{p^*} =S_{N,p}\|\na u_*\|_p$. For $\lambda >0$, define 
\begin{equation}\label{eq:newfunction}
w_\lam(x) = \frac{\lam^{\frac 1p} u_*(\lam^{\frac1N}x)}{\lt(\|u_*\|_p^\ga + \lam^{\frac\ga N} \|\na u_*\|_p^{\ga}\rt)^{\frac1\ga}}.
\end{equation}
Using the argument in the proof of \eqref{eq:bbbb} with $w_\lam$ as test function, we obtain
\begin{equation}\label{eq:dangthuc}
J_{N,p,\ga,\al}(w_\lam) = f_{N,p,\ga,\al}\lt(\lam^{\frac \ga N} \frac{\|\na u_*\|_p^\ga}{\|u_*\|_p^\ga}\rt).
\end{equation}

We first have the following estimates for $\sup_{t>0} f_{N,p,\ga,\al}(t)$ (and also for $D_{N,p,\ga,\al}$ by \eqref{eq:Dformulacritical}).
\begin{proposition}\label{chanduoiDcritical}
Let $N \geq 2$, $p\in (1,N)$ and $\al, \ga >0$. We have
\begin{equation}\label{eq:lowerboundD}
\sup_{t>0} f_{N,p,\ga,\al}(t) \geq \max\{1, \al S_{N,p}^{p^*}\}.
\end{equation}
Moreover, if $\ga > p^*$ then $\sup_{t>0} f_{N,p,\ga,\al}(t) > 1$ and if $\ga > p$ then $\sup_{t>0} f_{N,p,\ga,\al}(t) > \al S_{N,p,}^{p^*}$.
\end{proposition}
\begin{proof}
It is easy to see that
\[
\lim_{t\to 0} f_{N,p,\ga,\al}(t) = 1\qquad\text{\rm and } \qquad \lim_{t\to \infty} f_{N,p,\ga,\al}(t) = \al S_{N,p}^{p^*},
\]
which then imply \eqref{eq:lowerboundD}. If $\ga > p^*$, we have
\[
f_{N,p,\ga,\al}(t) = 1 -\frac p\ga t + \al S_{N,p}^{p^*} t^{\frac{p^*}\ga} + o(t)
\]
when $t \to 0$. Hence, for $t >0$ sufficiently small, we get
\[
\sup_{t>0} f_{N,p,\ga,\al}(t) \geq f_{N,p,\ga,\al}(t) > 1.
\]
If $\ga > p$, we have
\[
f_{N,p,\ga,\al}(t) = (1+ t)^{-\frac p\ga} + \al S_{N,p}^{p^*} -\al S_{N,p}^{p^*} \frac{p^*}\ga (1+ t)^{-1} + o((1+t)^{-1})
\]
when $t \to \infty$. Consequently, for $t >0$ sufficiently large, we get
\[
\sup_{t>0} f_{N,p,\ga,\al}(t) \geq f_{N,p,\ga,\al}(t) > \al S_{N,p}^{p^*}.
\]
\end{proof}

The rest of this section is devoted to prove Theorem \ref{Maximizerscritical}. Our proof is done by a series of lemmatas.
\begin{lemma}\label{nonexistence}
Let $N \geq 2$. If $p \geq N^{\frac12}$ then $D_{N,p,\al,\ga}$ does not attain for any $\al, \ga >0$.
\end{lemma}
\begin{proof}
Suppose that $D_{N,p,\al,\ga}$ is attained by a function $u\in S_{N,p,\ga}$, i.e.,
\[
D_{N,p,\ga,\al} = J_{N,p,\ga,\al}(u).
\]
Since $p \geq N^{1/2}$, then the Sobolev inequality \eqref{eq:Sobolevinequality} does not attains in $W^{1,p}(\R^N)$. Consequently, from the proof of \eqref{eq:aaaa}, we get
\[
J_{N,p,\ga,\al}(u) < f_{N,p,\ga,\al}\lt(\frac{\|\na u\|_p^\ga}{\|u\|_p^\ga}\rt).
\]
Therefore, we get
\[
D_{N,p,\ga,\al} = J_{N,p,\ga,\al}(u) < f_{N,p,\ga,\al}\lt(\frac{\|\na u\|_p^\ga}{\|u\|_p^\ga}\rt) \leq D_{N,p,\ga,\al},
\]
which is impossible. This finishes our proof.
\end{proof}
In the sequel, we only consider $p \in (1,N^{1/2})$. We have the following relation between the attainability of $D_{N,p,\ga,\al}$ and the attainability of $\sup_{t>0} f_{N,p,\ga,\al}(t)$ in $(0,\infty)$.
\begin{lemma}\label{relationcritical}
Let $N \geq 2$, $p \in (1,N^{1/2})$ and $\al, \ga > 0$. Then $D_{N,p,\ga,\al}$ is attained in $S_{N,p,\ga}$ if and only if $\sup_{t>0} f_{N,p,\ga,\al}(t)$ is attained in $(0,\infty)$.
\end{lemma}
\begin{proof}
Suppose that $D_{N,p,\ga,\al}$ is attained by a function $u\in S_{N,p,\ga,\al}$. It follows from \eqref{eq:Dformulacritical} and \eqref{eq:sosanh} that
\[
D_{N,p,\ga,\al} =J_{N,p,\ga,\al}(u) \leq f_{N,p,\ga,\al}\lt(\frac{\|\na u\|_p^\ga}{\|u\|_p^\ga}\rt) \leq D_{N,p,\ga,\al}.
\]
Consequently, we get
\[
f_{N,p,\ga,\al}\lt(\frac{\|\na u\|_p^\ga}{\|u\|_p^\ga}\rt) = D_{N,p,\ga,\al},
\]
that is $f_{N,p,\ga,\al}$ attains its maximum value at $\|\na u\|_p^\ga/ \|u\|_p^\ga \in (0, \infty)$.

Conversely, suppose that $\sup_{t>0} f_{N,p,\ga,\al}(t)$ is attained at $t_0\in (0,\infty)$. Choosing $\lam >0$ such that the function $w_\lam$ defined by \eqref{eq:newfunction} (note that $u_*\in W^{1,p}(\R^N)$ since $p < N^{1/2}$) satisfies 
$\lam^{\frac \ga N} \frac{\|\na u_*\|_p^\ga}{\|u_*\|_p^\ga} =t_0$. It follows from \eqref{eq:Dformulacritical} and \eqref{eq:dangthuc} that
\[
J_{N,p,\ga,\al}(w_\lam) = f_{N,p,\ga,\al}\lt(\lam^{\frac \ga N} \frac{\|\na u_*\|_p^\ga}{\|u_*\|_p^\ga}\rt) = f_{N,p,\ga,\al}(t_0) = \sup_{t >0} f_{N,p,\ga,\al}(t) =D_{N,p,\ga,\al},
\]
that is $D_{N,p,\ga,\al}$ is attained by $w_\lam$.
\end{proof}

We next prove a non-existence result when $\al < \al_{N,p}(\ga)$ as follows.
\begin{lemma}\label{kotontai}
Let $N \geq 2$, $p \in (1,N^{1/2})$ and $\ga > 0$. If $\al < \al_{N,p}(\ga)$ then $D_{N,p,\ga,\al}$  does not attains in $S_{N,p,\ga}$. Moreover $D_{N,p,\ga,\al} =1$.
\end{lemma}
\begin{proof}
Since $\al < \al_{N,p}(\ga)$ then $\al S_{N,p}^{p^*} t^{p^*/\ga} < (1+t)^{p^*/\ga} -(1+t)^{(p^*-p)/\ga}$ for any $t >0$. Hence
\[
f_{N,p,\ga,\al}(t) < 1 = \lim_{t\to 0} f_{N,p,\ga,\al}(t) \leq \sup_{t>0} f_{N,p,\ga,\al}(t)\qquad\forall\, t>0.
\]
These estimates implies that $\sup_{t>0} f_{N,p,\ga,\al}(t)$ does not attains in $(0,\infty)$ and hence $D_{N,p,\ga,\al}$ does not attains in $S_{N,p,\ga}$ by Lemma \ref{relationcritical}. Moreover, we get from the above estimates that $D_{N,p,\ga,\al} = \sup_{t>0} f_{N,p,\ga,\al}(t) =1$ by \eqref{eq:Dformulacritical}. 

\end{proof}
\begin{lemma}\label{case1}
Let $N \geq 2$, $p\in (1,N^{1/2})$ and $\al, \ga >0$. If $\ga > p^*$, then $\al_{N,p}(\ga) =0$ and $D_{N,p,\ga,\al}$ is attained for any $\al >0$. 
\end{lemma}
\begin{proof}
Since $\ga > p^*$, we then have $\lim\limits_{t\to 0} g_{N,p,\ga}(t) = 0$. Hence $\al_{N,p}(\ga) =0$. 

An easy computation shows that
\[
f'_{N,p,\ga,\al}(t) = \frac{p}{\ga(1+t)^2}\lt(-(1+t)^{1-\frac p\ga} + \frac{p^*}p \al S_{N,p}^{p^*} \lt(\frac t{t+1}\rt)^{\frac{p^*}\ga -1}\rt)=: \frac{p}{\ga(1+t)^2} h_{N,p,\ga,\al}(t).
\]
Since $\ga > p^*$, then $h_{N,p,\ga,\al}$ is strictly decreasing function on $(0,\infty)$ with $\lim\limits_{t\to 0} h_{N,p,\ga,\al}(t) = \infty$ and $\lim\limits_{t\to \infty} h_{N,p,\ga,\al}(t) = -\infty$. There exists unique $t_0\in (0,\infty)$ such that $h_{N,p,\ga,\al}(t_0) =0$, $h_{N,p,\ga,\al} > 0$ on $(0,t_0)$ and $h_{N,p,\ga,\al} < 0$ on $(t_0,\infty)$. $h_{N,p,\ga,\al}$ and $f_{N,p,\ga,\al}'$ has the same sign on $(0,\infty)$ which implies $f_{N,p,\ga,\al}$ attains it maximum value at unique point $t_0\in (0,\infty)$. Our conclusion hence is followed from Lemma \ref{relationcritical}.

\end{proof}
\begin{lemma}\label{case2}
Let $N \geq 2$, $p\in (1,N^{1/2})$ and $\al, \ga >0$. If $\ga = p^*$, then $\al_{N,p}(p^*) =p/(p^* S_{N,p}^{p^*})$ and $D_{N,p,\ga,\al}$ is attained for any $\al > \al_{N,p}(p^*)$ while it is not attained for any $\al \leq \al_{N,p}(p^*)$. 
\end{lemma}
\begin{proof}
Suppose that $\ga = p^*$, we then have
\[
g_{N,p,p^*}(t) = \frac{1+ t} t -\frac{(1+t)^{1-\frac p{p^*}}}{t},
\]
and hence
\begin{align*}
g_{N,p,p^*}'(t) &= -\frac1 {t^2} - \lt(1 -\frac p{p^*}\rt) \frac{(1+t)^{-\frac p{p^*}}}t + \frac{(1+ t)^{1 -\frac p{p^*}}}{t^2}\\
&=\frac{1}{t^2} \lt(\frac p{p^*} (1+t)^{1-\frac p{p^*}} + \lt(1 -\frac p{p^*}\rt) (1+t)^{-\frac p{p^*}} -1\rt).
\end{align*}
By the arithmetic--geometric inequality we get $g_{N,p,p^*}'(t) > 0$ for any $t >0$. This implies that $g_{N,p,p^*}$ is strictly increasing function on $(0,\infty)$ and hence
\begin{equation}\label{eq:cccc}
\al_{N,p}(p^*) =\frac1{S_{N,p}^{p^*}}\lim_{t\to \infty} g_{N,p,p^*}(t) = \frac{p}{p^* S_{N,p}^{p^*}}\,\,\,\text{\rm and }\,\, \al_{N,p}(p^*)S_{N,p}^{p^*}<  g_{N,p,p^*}(t)\quad\forall\, t>0.
\end{equation}

If $\al > \al_{N,p}(p^*)$, then we can take $t_1\in (0,\infty)$ such that $\al S_{N,p}^{p^*} > g_{N,p,p^*}(t_1)$ which implies $\sup_{t>0} f_{N,p,p^*,\al}(t) \geq f_{N,p,p^*,\al}(t_1) > 1$. Proposition \ref{chanduoiDcritical} yields $\sup_{t>0} f_{N,p,p^*,\al}(t) > \al S_{N,p}^{p^*}$ since $\ga =p^* > p$. Hence $f_{N,p,p^*,\al}$ attains it maximum value in $(0,\infty)$. By Lemma \ref{relationcritical}, $D_{N,p,p^*,\al}$ is attained in $S_{N,p,p^*}$.

We next consider the case $\al \leq \al_{N,p}(p^*)$. It remains to prove for $\al = \al_{N,p}(p^*)$ by Lemma \ref{kotontai}. It follows from \eqref{eq:cccc} that $\al_{N,p}(p^*) S_{N,p}(p^*) < g_{N,p,p^*}(t)$ for any $t>0$ which implies $f_{N,p,p^*, \al_{N,p}(p^*)}(t) < 1 \leq \sup_{t>0} f_{N,p,p^*,\al_{N,p}(p^*)}(t)$ for any $t >0$ by Proposition \ref{chanduoiDcritical}. Consequently, $\sup_{t>0} f_{N,p,p^*,\al_{N,p}(p^*)}(t)$ is not attained in $(0,\infty)$. By Lemma \ref{relationcritical}, $D_{N,p,p^*,\al_{N,p}(p^*)}$ is not attained. Moreover, the previous estimate and \eqref{eq:Dformulacritical} yields $D_{N,p,p^*,\al_{N,p}(p^*)} =1$.
\end{proof}

\begin{lemma}\label{case3}
Let $N \geq 2$, $p\in (1,N^{1/2})$ and $\al, \ga >0$. If $p < \ga <p^*$, then $D_{N,p,\ga,\al}$ is attained for any $\al \geq \al_{N,p}(\ga)$ while it is not attained for any $\al < \al_{N,p}(\ga)$. Moreover, we have $D_{N,p,\ga,\al} =1$ if $\al \leq \al_{N,p}(\ga)$.
\end{lemma}
\begin{proof}
By Lemma \ref{kotontai}, it remains to consider the case $\al \geq \al_{N,p}(p^*)$. Making the change of variable $s = t/(1+t) \in (0,1)$, we get
\[
f_{N,p,\ga,\al}(t) = (1-s)^{\frac{p}\ga}  + \al S_{N,p}^{p^*} s^{\frac{p^*}\ga} = :k_{N,p,\ga,\al}(s),
\]
and
\[
g_{N,p,\ga}(t) = s^{-\frac{p^*}\ga} -s^{-\frac{p^*}\ga}(1-s)^{\frac p\ga} = :l_{N,p,\ga}(s),
\]
with $s\in (0,1)$. We have 
\begin{align*}
l'_{N,p,\ga}(s) &= -\frac{p^*}\ga s^{-\frac {p^*}\ga -1} + \frac{p^*}\ga s^{-\frac {p^*}\ga -1}(1-s)^{\frac p\ga} + \frac p\ga s^{-\frac{p^*}\ga}(1-s)^{\frac p\ga -1} \\
&=\frac{p^*}\ga s^{-\frac {p^*}\ga -1} \lt(\lt(1 -\frac p{p^*}\rt) (1-s)^{\frac p\ga} + \frac p{p^*}(1-s)^{\frac p\ga -1} -1\rt)\\
&=: \frac{p^*}\ga s^{-\frac {p^*}\ga -1}\, m_{N,p,\ga}(s),
\end{align*}
and
\[
m_{N,p,\ga}'(s)= -\frac p\ga (1-s)^{\frac p\ga -2} \lt(\frac{p^*-\ga}{p^*} - \frac{p^*-p}{p^*} s\rt). 
\]
Consequently, $m'_{N,p,\ga} =0$ at $s_1 =(p^* -\ga)/(p^*-p) \in (0,1)$, $m'_{N,p,\ga} < 0$ on $(0,s_1)$ and $m'_{N,p,\ga} >0$ on $(s_1,1)$. Since $\lim_{s\to 0} m_{N,p,\ga}(s) = 0$, then there exists unique $s_2 \in (s_1,1)$ such that $m_{N,p,\ga}(s_2) =0$, $m_{N,p,\ga} < 0$ on $(0,s_2)$ and $m_{N,p,\ga} >0$ on $(s_2,1)$. $l'_{N,p,\ga}$ and $m_{N,p,\ga}$ have the same sign, hence $l_{N,p,\ga}$ attains uniquely its minimum value at $s_2$ which is equivalent that $g_{N,p,\ga}$ attains uniquely its minimum value at $t_2 =s_2/(1-s_2) > s_1/(1-s_1) = (p^*-\ga)/(\ga -p)$.

If $\al > \al_{N,p}(\ga) = S_{N,p}^{-p^*} l_{N,p,\ga}(s_2)$, then we have
\begin{equation}\label{eq:1111}
k_{N,p,\ga,\al}(s_2) = (1-s_2)^{\frac{p}\ga}  + \al S_{N,p}^{p^*} s_2^{\frac{p^*}\ga} \geq 1 = \lim_{s\to 0} k_{N,p,\ga,\al}(s).
\end{equation}
Let $k\geq 1$ denote the integer number such that $k < p^*/ \ga \leq k+1$, i.e, $p^*/\ga = k + r$ for $r \in (0,1]$. The $i-$th derivative of $k_{N,p,\ga,\al}$ is given by
\[
k_{N,p,\ga,\al}^{(i)}(s) = \prod_{j=0}^{i-1}\lt(j -\frac p\ga\rt)(1- s)^{\frac p\ga -i} + \al S_{N,p}^{p^*} \prod_{j=0}^{i-1} \lt(\frac{p^*}\ga -j\rt) s^{\frac {p^*}\ga -i}.
\]
For $i=1,2,\ldots,k$ we have
\begin{equation}\label{eq:daohambaci}
\lim_{s\to 0}k_{N,p,\ga,\al}^{(i)}(s) =\prod_{j=0}^{i-1}\lt(j -\frac p\ga\rt) < 0,\qquad \lim_{s\to 1} k_{N,p,\ga,\al}^{(i)}(s) =-\infty.
\end{equation}
Combining \eqref{eq:1111} and \eqref{eq:daohambaci}, we get that
\begin{equation}\label{eq:nega}
\sup_{s\in (0,1)} k_{N,p,\ga,\al}^{(i)}(s) >0\qquad i =1,2,\ldots,k.
\end{equation}
Indeed, if this is not true, i.e., there is $i$ such that $\sup_{s\in (0,1)} k_{N,p,\ga,\al}^{(i)}(s) \leq 0$, then using \eqref{eq:daohambaci} consecutively we see that $k_{N,p,\ga,\al}$ is strictly decreasing on $(0,1)$ which contradicts with \eqref{eq:1111}.

The functions $ k_{N,p,\ga,\al}^{(k+1)}$ is strictly decreasing on $(0,1)$ with
\[
\lim_{s\to 0}k_{N,p,\ga,\al}^{(k+1)}(s) \geq \prod_{j=0}^{k} \lt(\frac{p^*}\ga -j\rt) > 0,\qquad \lim_{s\to 1} k_{N,p,\ga,\al}^{(i)}(s) =-\infty.
\]
Hence there exists uniquely $\tilde s_{k+1} \in (0,1)$ such that $k_{N,p,\ga,\al}^{(k+1)}(\tilde s_{k+1}) =0$, $k_{N,p,\ga,\al}^{(k+1)} >0$ on $(0,\tilde s_{k+1})$ and $k_{N,p,\ga,\al}^{(k+1)} < 0$ on $(\tilde s_{k+1},1)$, i.e., $k_{N,p,\ga,\al}^{(k)}$ attains its maximum value at $\tilde s_{k+1}$.  Hence $k_{N,p,\ga,\al}^{(k)}(\tilde s_{k+1}) > 0$. The function $k_{N,p,\ga,\al}^{(k)}$ is strictly increasing on $(0,\tilde s_{k+1})$ and strictly decreasing on $(\tilde s_{k+1},1)$, hence there are two points $\tilde s_k' \in (0,\tilde s_{k+1})$ and $\tilde s_k'' \in (\tilde s_{k+1},1)$ such that $k_{N,p,\ga,\al}^{(k)}(\tilde s_k') = k_{N,p,\ga,\al}^{(k)}(\tilde s_k'') =0$, $k_{N,p,\ga,\al}^{(k)} < 0$ on $(0,\tilde s_k') \cup (\tilde s_k'',1)$ and $k_{N,p,\ga,\al}^{(k)} >0$ on $(\tilde s_k',\tilde s_k'')$. This together with \eqref{eq:daohambaci} implies that $k_{N,p,\ga,\al}^{(k-1)}$ attains its maximum value only at $\tilde s_k''$. Obviously, the above argument for $k_{N,p,\ga,\al}^{(k)}$ shows that $k_{N,p,\ga,\al}^{(k-1)}(\tilde s_k'') >0$, and $k_{N,p,\ga,\al}^{(k-1)}$ has only two solutions $\tilde s_{k-1}' \in (\tilde s_k',\tilde s_k'')$ and $\tilde s_{k-1}'' \in (\tilde s_k'',1)$ such that $k_{N,p,\ga,\al}^{(k-1)} < 0$ on $(0,\tilde s_{k-1}') \cup (\tilde s_{k-1}'',1)$ and $k_{N,p,\ga,\al}^{(k-1)} >0$ on $(\tilde s_{k-1}',\tilde s_{k-1}'')$. Using consecutively this arguement for $i=k-2, \ldots 1$, we conclude that $k_{N,p,\ga,\al}'$ has only two solutions $\tilde s_{1}' \in (\tilde s_2',\tilde s_2'')$ and $\tilde s_{1}'' \in (\tilde s_2'',1)$ such that $k_{N,p,\ga,\al}' < 0$ on $(0,\tilde s_1') \cup (\tilde s_1'',1)$ and $k_{N,p,\ga,\al}^{(k-1)} >0$ on $(\tilde s_1',\tilde s_1'')$. This together with \eqref{eq:1111} and Proposition \ref{chanduoiDcritical} implies that $k_{N,p,\ga,\al}$ attains uniquely its maximum value at $s_1'' \in (0,1)$.

Suppose $\al = \al_{N,p}(\ga)$. We have $\al_{N,p}(\ga) S_{N,p}^{p^*} < g_{N,p,\ga}(t)$ unless $t =t_2$ which immediately implies $f_{N,p,\ga,\al_{N,p}(\ga)}(t) < 1$ unless $t =t_2$. This together with Proposition \ref{chanduoiDcritical} gives
\[
D_{N,p,\ga,\al_{N,p}(\ga)} = \sup_{t>0} f_{N,p,\ga,\al_{N,p}(\ga)}(t) = 1,
\]
and $\sup_{t>0} f_{N,p,\ga,\al_{N,p}(\ga)}(t)$ is attained only at $t_2 \in (0,1)$.
\end{proof}
Let us remark here that when $\al \geq \al_{N,p}(\gamma)$ then
\[
\sup_{t> 0} f_{N,p,\ga,\al}(t) \geq f_{N,p,\ga,\al}(t_2) \geq f_{N,p,\ga,\al_{N,p}(\ga)}(t_2) = 1 = \lim_{t\to 0} f_{N,p,\ga}(t).
\]
This together with Proposition \ref{chanduoiDcritical} is enough to conclude that $\sup_{t>0} f_{N,p,\ga,\al_{N,p}(\ga)}(t)$ is attained on $(0,\infty)$. Although our proof above is more complicate, however it shows that $\sup_{t>0} f_{N,p,\ga,\al_{N,p}(\ga)}(t)$ is attained at unique point in $(0,\infty)$. This will be used to classify all maximizers for $D_{N,p,\al,\ga}$ below.

\begin{lemma}\label{case4}
Let $N \geq 2$, $p\in (1,N^{1/2})$ and $\al, \ga >0$. If $0 < \ga \leq p$, then $\al_{N,p}(\ga) = S_{N,p}^{-p^*}$ and $D_{N,p,\ga,\al}$ is not attained for any $\al >0$. Moreover, we have $ D_{N,p,\ga,\al} = \max\{1, \al S_{N,p}^{p^*}\}$..
\end{lemma}
\begin{proof}
Making the change of variable $s = t/(1+t) \in (0,1)$ and using again the transformation of functions as in the proof of Lemma \ref{case3} we have $l'_{N,p,\ga}(s) = \frac{p^*}\ga s^{-\frac {p^*}\ga -1}\, m_{N,p,\ga}(s)$.
Since $\ga \leq p$, it is clear that $m_{N,p,\ga}$ is a strictly decreasing function of $s$ on $(0,1)$ and $\lim\limits_{s\to 0} m_{N,p,\ga}(s) =0$. Hence, $m_{N,p,\ga}(s) < 0$ for any $s \in (0,1)$ or equivalently $l'_{N,p,\ga}(s) < 0$ for any $s\in (0,1)$. Consequently, we have 
\[
\inf_{t> 0} g_{N,p,\ga}(t) = \inf_{s \in (0,1)} l_{N,p,\ga}(s) = \lim_{s\to 1}l_{N,p,\ga}(s) =1.
\] 
Hence $\al_{N,p}(\ga) = S_{N,p}^{-p^*}$.

We extend $k_{N,p,\ga,\al}$ to all $[0,1]$ by 
\[
k_{N,p,\ga,\al}(0) = \lim_{s\to 0} k_{N,p,\ga,\al}(s) =1, \qquad\text{\rm and }\qquad k_{N,p,\ga,\al}(1) =\lim_{s\to 1} k_{N,p,\ga,\al}(s) = \al S_{N,p}^{p^*}.
\]
Then $k_{N,p,\ga,\al}$ is continuous on $[0,1]$. Since $\ga \leq p < p^*$ then $k_{N,p,\ga,\al}$ is strict convex on $[0,1]$. As a consequence, we have for any $s \in (0,1)$ that
\[
k_{N,p,\ga,\al}(s) < s k_{N,p,\ga,\al}(1) + (1-s) k_{N,p,\ga,\al}(0) \leq \max\{1,\al S_{N,p}^{p^*}\}.
\]
Hence, for any $t >0$ we get
\[
f_{N,p,\ga,\al}(t) = k_{N,p,\ga,\al}\lt(\frac{t}{1+t}\rt) < \max\{1,\al S_{N,p}^{p^*}\}.
\]
The preceding inequality together with Proposition \ref{chanduoiDcritical} implies 
\[
D_{N,p,\ga,\al} = \sup_{t>0} f_{N,p,\ga,\al}(t) = \max\{1,\al S_{N,p}^{p^*}\}.
\]
Moreover, we know that $f_{N,p,\ga,\al}(t) < \sup_{t>0} f_{N,p,\ga,\al}(t)$ for any $t >0$. Hence $\sup_{t>0} f_{N,p,\ga,\al}(t)$ does not attains in $(0,\infty)$ which finishes our proof by Lemma \ref{relationcritical}.
\end{proof}
Concerning to $\al_{N,p}(\ga)$, we have the following results on its attainability and its non-attainability in $S_{N,p,\ga}$.
\begin{proposition}
Let $N\geq 2$ and $p\in (1,N)$. If $p\geq N^{1/2}$ then $\al_{N,p}(\ga)$ is not attained for any $\ga >0$. If $p < N^{1/2}$ then $\al_{N,p}(\ga)$ is attained for any $p < \ga < p^*$ while it is not attained for any $\ga \geq p^*$ or $\ga \leq p$. Moreover, the function $\ga \mapsto \al_{N,p}(\ga)$ is continuous on $(0,p^*]$.
\end{proposition}
\begin{proof}
The non-attainability of $\al_{N,p}(\ga)$ for $p\geq N^{1/2}$ is completely proved by the method in the proof of Lemma \ref{nonexistence} by exploiting the non-attainability of the Sobolev inequality \eqref{eq:Sobolevinequality} in $W^{1,p}(\R^N)$. The attainability of $\al_{N,p}(\ga)$ for $p < N^{1/2}$ and $\ga\in (p,p^*)$ was already proved in the proof of Lemma \ref{case3}. The non-attainability of $\al_{N,p}(\ga)$ for $p < N^{1/2}$ and $\gamma \not\in (p,p^*)$ follows from the proof of Lemma \ref{case1}, \ref{case2} and \ref{case4}.

To prove the continuity of $\al_{N,p}$ on $(0,p^*]$, it is enough to show that it is continuous on $[p,p^*]$ since $\al_{N,p}(\ga) = S_{N,p}^{-p^*}$ for any $\ga \leq p$. We claim that
\begin{equation}\label{eq:semicont}
\limsup_{\ga \to \ga_0} \al_{N,p}(\ga) \leq \al_{N,p}(\ga_0).
\end{equation}
This is an easy consequence of the continuity of $(0,\infty)\times (0,\infty) \ni (t,\ga) \mapsto g_{N,p,\ga}(t)$. Rewriting the function $g_{N,p,\ga}(t)$ as
\[
g_{N,p,\ga}(t) = \lt(\frac{1+t}t\rt)^{\frac{p^*}\ga} \lt(1 - \lt(\frac1{1+t}\rt)^{\frac p\ga}\rt)
\]
we see that $\ga \mapsto g_{N,p,\ga}(t)$ is a strictly decreasing function of $\ga$. Hence $\ga \mapsto \al_{N,p}(\ga)$ is a non-increasing function of $\ga >0$ and is strictly decreasing function of $\ga \in (p,p^*)$ since $\al_{N,p}(\ga)$ is attained for $\ga \in (p,p^*)$. Then non-increasing property of $\al_{N,p}(\ga)$ and \eqref{eq:semicont} imply
\[
\lim_{\ga \uparrow p^*} \al_{N,p}(\ga) = \al_{N,p}(p^*).
\]
Let $t_\ga \in (0,\infty)$ be the unique point where $g_{N,p,\ga}$ attains its minimum value in $(0,\infty)$ (the uniqueness come from the proof of Lemma \ref{case3}). Also, from the proof of Lemma \ref{case3}, we get $t_\ga > (p^*-\ga)/(\ga -p)$, hence $t_\ga \to \infty$ as $\ga \to p$. Hence
\[
\lim_{\ga \to p} \al_{N,p}(\ga) = \frac1{S_{N,p}^{p^*}} \lim_{\ga\to p} \lt(\frac{1+t_\ga}{t_\ga}\rt)^{\frac{p^*}\ga} \lt(1 - \lt(\frac1{1+t_\ga}\rt)^{\frac p\ga}\rt) =S_{N,p}^{-p^*} = \al_{N,p}(p).
\]
Hence $\al_{N,p}$ is continuous at $p$ and $p^*$. For any $p < a <  b < p^*$ and $\ga \in [a,b]$ we have 
\[
\lim_{t\to 0} g_{N,p,\ga}(t) \geq \lim_{t\to 0} g_{N,p,b}(t) = \infty > \inf_{t>0} g_{N,p,a}(t),
\]
and
\[
\lim_{t\to \infty} g_{N,p,\ga}(t)\geq \lim_{t\to \infty} g_{N,p,b}(t) =1 > \inf_{t>0} g_{N,p,a}(t),
\]
since $a > p$. Hence there exist $0 < A < B< \infty$ such that 
\[
g_{N,p,\ga}(t) > \inf_{t>0} g_{N,p,a}(t) > \inf_{t>0} g_{N,p,\ga}(t)
\]
for any $t\in (0,A)\cup (B,\infty)$ and $\ga \in [a,b]$. Thus, we get
\[
\al_{N,p}(\ga) = S_{N,p}^{-p^*} \inf_{t\in [A,B]} g_{N,p,\ga}(t),
\]
for any $\ga \in [a,b]$. According to the continuity of function $[a,b]\times [A,B] \ni (\ga,t) \mapsto g_{N,p,\ga}(t)$, the function $\ga \mapsto\al_{N,p}(\ga)$ is continuous on $[a,b]$ for any $p < a < b < p^*$. Hence $\al_{N,p}$ is continuous on $(p,p^*)$. This completes our proof. 
\end{proof}


To conclude this section, we give some comments on the maximizers for the problem \eqref{eq:variationalproblem} in the critical case. From the proof of Theorem \ref{Maximizerscritical} (more precisely, Lemma \ref{relationcritical}), we see that if the maximizers of the problem \eqref{eq:variationalproblem} in the critical case exists then $w_\lam$ defined by \eqref{eq:newfunction} for some suitable $\lam >0$ determined by the value of $D_{N,p,\ga,\al}$ is a maximizer. Since our problem is invariant under the translation and multiple by $\pm 1$ then $\pm w_\lam(\cdot-x_0), x_0\in \R^N$ is also a maximizer. Conversely, if $u$ is a maximizer of $D_{N,p,\ga,\al}$, i.e., 
\[
\sup_{t>0} f_{N,p,\ga,\al}(t) = D_{N,p,\ga,\al} = J_{N,p,\ga,\al}(u) \leq f_{N,p,\ga,\al}\lt(\frac{\|\na u\|_p^\ga}{\|u\|_p^\ga}\rt)\leq \sup_{t>0} f_{N,p,\ga,\al}(t),
\]
here we used \eqref{eq:sosanh}. This shows that $f_{N,p,\ga,\al}$ is attained by $\|\na u\|_p^\ga/\|u\|_p^\ga$ and
\[
J_{N,p,\ga,\al}(u) \leq f_{N,p,\ga,\al}\lt(\frac{\|\na u\|_p^\ga}{\|u\|_p^\ga}\rt).
\]
The latter condition implies that $u$ is a maximizer of the Sobolev inequality \eqref{eq:Sobolevinequality}, i.e., $u$ has the form 
$u(x) = a \lam^{\frac 1p} u_*(\lam^{\frac1N}(x-x_0))$ for some $a\not=0$, $\lam >0$ and $x_0 \in \R^N$ where $u_*$ is given by \eqref{eq:optimalSobolev}. To ensure $u\in S_{N,p,\ga}$, $a$ must be equal to $
a = \pm 1/\lt(\|u_*\|_p^\ga + \lam^{\frac \ga N} \|\na u_*\|_p^\ga\rt)^{\frac1\ga}$. Hence $u(x) =\pm w_{\lam}(x-x_0)$ for some $x_0 \in \R^N$ and $\lam >0$. Let $t_{N,p,\ga,\al} \in (0,\infty)$ be the unique point where $f_{N,p,\ga,\al}$ attains its minimum value, then we must have $\|\na u\|_p^\ga/\|u\|_p^\ga =t_{N,p,\ga,\al}$, which is equivalent to 
\[
\lam = t_{N,p,\ga,\al}^{\frac N\ga} \lt(\frac{\|u_*\|_p}{\|\na u_*\|_p}\rt)^N.
\]
Therefore $\lam$ is determined uniquely by $t_{N,p,\ga,\al}$ (or equivalent $D_{N,p,\ga,\al}$ since it is attained at unique point in $(0,\infty)$). Thus, we have proved the following result.
\begin{proposition}\label{eq:classify}
Let $N \geq 2$, $p\in (1,N^{\frac12})$ and $\al, \ga >0$. If $D_{N,p,\ga,\al}$ is attained in $S_{N,p,\ga}$, then its maximizers must have the form $\pm\lam^{\frac1p} u_*(\lam^{\frac1N}(x-x_0))/(\|u_*\|_p^\ga + \lam^{\frac\ga N} \|\na u_*\|_p^\ga)^{\frac 1\ga}$ for some $x_0 \in \R^N$ where $\lam >0$ is determined uniquely by $D_{N,p,\ga,\al}$.
\end{proposition}



\section{A new proof of Theorem \ref{IshiwataWadade}}
This section is devoted to give a new proof of Theorem \ref{IshiwataWadade}. Comparing with the proof of Ishiwata and Wadade in \cite{Ishiwata}, our proof is more simple and elementary. Also, our proof does not use any rearrangement argument as used in \cite{Ishiwata}. We first prove Theorem \ref{explicitformula} which is the key in our proof of Theorem \ref{IshiwataWadade}.

\begin{proof}[Proof of Theorem \ref{explicitformula}]
Our proof follows the way in the proof of Theorem \ref{formulacritical}. We start by proving \eqref{eq:Dformula}. For any $u\in S_{N,p,\ga}$, using Gagliardo--Nirenberg--Sobolev inequality \eqref{eq:GNSinequality}, we get
\begin{align*}
J_{N,p,\ga,\al,q}(u) &\leq \|u\|_p^p + \al B_{N,p,q} \|\na u\|_p^{\ga_c} \|u\|_p^{q -\ga_c} \\
&=\frac{\|u\|_p^p \lt(\|u\|_p^\ga + \|\na u\|_p^\ga\rt)^{\frac{q-p}\ga} + \al B_{N,p,q} \|\na u\|_p^{\ga_c} \|u\|_p^{q -\ga_c}}{\lt(\|u\|_p^\ga + \|\na u\|_p^\ga\rt)^{\frac q\ga}}\\
&=\frac{\lt(1 + \frac{\|\na u\|_p^\ga}{\|u\|_p^\ga}\rt)^{\frac{q-p}\ga} + \al B_{N,p,q} \lt(\frac{\|\na u\|_p^\ga}{\|u\|_p^\ga}\rt)^{\frac{\ga_c}\ga}}{\lt(1 + \frac{\|\na u\|_p^\ga}{\|u\|_p^\ga}\rt)^{\frac{q}\ga}}\\
&\leq \sup_{t >0}\, \frac{(1+t)^{\frac{q-p}\ga} + \al B_{N,p,q} t^{\frac{\ga_c}\ga}}{(1+t)^{\frac q\ga}}.
\end{align*}
Taking the supremum over all $u\in S_{N,p,\ga}$, we obtain
\begin{equation}\label{eq:nhohon}
D_{N,p,\ga,\al,q} \leq \sup_{t >0}\, \frac{(1+t)^{\frac{q-p}\ga} + \al B_{N,p,q} t^{\frac{\ga_c}\ga}}{(1+t)^{\frac q\ga}}.
\end{equation}

For a reversed inequality of \eqref{eq:nhohon}, let us take $v_*\in W^{1,p}(\R^N)\setminus\{ 0\}$ to be a maximizer for the Gagliardo--Nirenberg--Sobolev inequality \eqref{eq:GNSinequality}. For any $\lam >0$, we define
\begin{equation}\label{eq:testsequence}
v_\lam(x) = \frac{\lam^{\frac1p} v_*(\lam^{\frac1N} x)}{\lt(\|v_*\|_p^\ga + \lam^{\frac \ga N} \|\na v_*\|_p^\ga\rt)^{\frac1\ga}}.
\end{equation}
We have $v_\lam \in S_{N,p,\ga}$ and
\begin{align}\label{eq:applytowlambda}
D_{N,p,\ga,\al,q} \geq J_{N,p,\ga,\al,q}(v_\lam)&= \frac{\|v_*\|_p^p}{\lt(\|v_*\|_p^\ga + \lam^{\frac \ga N} \|\na v_*\|_p^\ga\rt)^{\frac p\ga}} + \al \frac{\lam^{\frac{q-p}p} \|v_*\|_{q}^q}{\lt(\|v_*\|_p^\ga + \lam^{\frac \ga N} \|\na v_*\|_p^\ga\rt)^{\frac q\ga}}\notag\\
&=\frac{\|v_*\|_p^p}{\lt(\|v_*\|_p^\ga + \lam^{\frac \ga N} \|\na v_*\|_p^\ga\rt)^{\frac p\ga}} + \al B_{N,p,q} \frac{\lam^{\frac{q-p}p} \|\na v_*\|_p^{\ga_c}\|v_*\|_p^{q-\ga_c}}{\lt(\|v_*\|_p^\ga + \lam^{\frac \ga N} \|\na v_*\|_p^\ga\rt)^{\frac q\ga}}\notag\\
&=\frac{\lt(1 + \lam^{\frac\ga N} \frac{\|\na v_*\|_p^\ga}{\|v_*\|_p^\ga}\rt)^{\frac{q-p}\ga} + \al B_{N,p,q} \lt(\lam^{\frac\ga N} \frac{\|\na v_*\|_p^\ga}{\|v_*\|_p^\ga}\rt)^{\frac{\ga_c}\ga}}{\lt(1 + \lam^{\frac\ga N} \frac{\|\na v_*\|_p^\ga}{\|v_*\|_p^\ga}\rt)^{\frac{q}\ga}}.
\end{align}
Since \eqref{eq:applytowlambda} holds for any $\lam >0$, then we get
\begin{equation}\label{eq:lonhon}
D_{N,p,\ga,\al,q} \geq \sup_{t >0}\, \frac{(1+t)^{\frac{q-p}\ga} + \al B_{N,p,q} t^{\frac{\ga_c}\ga}}{(1+t)^{\frac q\ga}}.
\end{equation}
Combining \eqref{eq:nhohon} and \eqref{eq:lonhon} together, we obtain \eqref{eq:Dformula}.

We next prove the inequality \eqref{eq:alphaformula}. For any $u \in S_{N,p,\ga}$, by using the sharp Gagliardo--Nirenberg--Sobolev inequality \eqref{eq:GNSinequality}, we get
\begin{align*}
I_{N,p,\ga,q}(u) &\geq \frac1{B_{N,p,q}}\, \frac{1 -\|u\|_p^p}{\|\na u\|_p^{\ga_c} \|u\|_p^{q-\ga_c}} \\
&= \frac1{B_{N,p,q}}\, \frac{\lt(\|u\|_p^\ga + \|\na u\|_p^\ga\rt)^{\frac{q-p}\ga} \lt(\lt(\|u\|_p^\ga + \|\na u\|_p^\ga\rt)^{\frac p\ga} -\|u\|_p^p\rt)}{\|\na u\|_p^{\ga_c} \|u\|_p^{q-\ga_c}}\\
&=\frac1{B_{N,p,q}}\, \frac{\lt(1 + \frac{\|\na u\|_p^\ga}{\|u\|_p^\ga}\rt)^{\frac{q-p}\ga} \lt(\lt(1 + \frac{\|\na u\|_p^\ga}{\|u\|_p^\ga}\rt)^{\frac p\ga} -1\rt)}{\lt(\frac{\|\na u\|_p^\ga}{\|u\|_p^\ga}\rt)^{\frac {\ga_c}\ga}}\\
&\geq \frac1{B_{N,p,q}}\,\inf_{t >0}\, \frac{(1+t)^{\frac{q-p}\ga} \lt((1+t)^{\frac p\ga} -1\rt)}{t^{\frac{\ga_c}\ga}}.
\end{align*}
Taking the infimum over all $u \in S_{N,p,\ga}$ we obtain
\begin{equation}\label{eq:lonhon1}
\al_{N,p,q}(\ga) \geq \frac1{B_{N,p,q}}\,\inf_{t >0}\, \frac{(1+t)^{\frac{q-p}\ga} \lt((1+t)^{\frac p\ga} -1\rt)}{t^{\frac{\ga_c}\ga}}.
\end{equation}

To finish the prove of \eqref{eq:alphaformula}, let us prove a reversed inequality of \eqref{eq:lonhon1}. We use again the function $v_\lam$ defined by \eqref{eq:testsequence} as test function for $\al_{N,p,q}(\ga)$. Obviously, we have
\begin{align}\label{eq:applytowlambda1}
\al_{N,p,q}(\ga) &\leq I_{N,p,\ga,q}(v_\lam)\notag \\
&= \lt(1 -\frac{\|v_*\|_p^p}{\lt(\|v_*\|_p^\ga + \lam^{\frac \ga N}\|\na v_*\|_p^\ga\rt)^{\frac p\ga}}\rt)\lt(\frac{\lam^{\frac{q-p}\ga} \|v_*\|_q^q}{\lt(\|v_*\|_p^\ga + \lam^{\frac \ga N}\|\na v_*\|_p^\ga\rt)^{\frac q\ga}}\rt)^{-1}\notag\\
&=\frac1{B_{N,p,q}} \,\frac{\lt(\|v_*\|_p^\ga + \lam^{\frac \ga N}\|\na v_*\|_p^\ga\rt)^{\frac{q-p}\ga}\lt(\lt(\|v_*\|_p^\ga + \lam^{\frac \ga N}\|\na v_*\|_p^\ga\rt)^{\frac p\ga} -\|v_*\|_p^p\rt)}{\lam^{\frac{q-p}\ga} \|\na v_*\|_p^{\ga_c} \|u\|_p^{q-\ga_c}}\notag\\
&=\frac1{B_{N,p,q}}\,\frac{\lt(1 + \frac{\lam^{\frac \ga N} \|\na v_*\|_p^\ga}{\|v_*\|_p^\ga}\rt)^{\frac{q-p}\ga} \lt(\lt(1 + \frac{\lam^{\frac \ga N} \|\na v_*\|_p^\ga}{\|v_*\|_p^\ga}\rt)^{\frac{p}\ga} -1\rt)}{\lt(\frac{\lam^{\frac \ga N} \|\na v_*\|_p^\ga}{\|v_*\|_p^\ga}\rt)^{\frac{\ga_c}\ga}},
\end{align}
here we use the equality in \eqref{eq:GNSinequality} for $v_*$. The inequality \eqref{eq:applytowlambda1} holds for any $\lambda >0$, then we have 
\begin{equation}\label{eq:nhohon1}
\al_{N,p,q}(\ga) \leq \frac1{B_{N,p,q}}\,\inf_{t >0}\, \frac{(1+t)^{\frac{q-p}\ga} \lt((1+t)^{\frac p\ga} -1\rt)}{t^{\frac{\ga_c}\ga}}.
\end{equation}
Combining \eqref{eq:nhohon1} and \eqref{eq:lonhon1} together, we obtain \eqref{eq:alphaformula}.
\end{proof}

The above proof of Theorem \ref{explicitformula} shows that
\begin{equation}\label{eq:relation}
J_{N,p,\ga,\al,q}(u) \leq f_{N,p,\ga,\al,q}\lt(\frac{\|\na u\|_p^\ga}{\|u\|_p^\ga}\rt),\qquad I_{N,p,\ga,q}(u)\geq \frac{1}{B_{N,p,q}}\, g_{N,p,\ga,q}\lt(\frac{\|\na u\|_p^\ga}{\|u\|_p^\ga}\rt),
\end{equation}
and 
\begin{equation}\label{eq:dangthucsub}
J_{N,p,\ga,\al,q}(v_\lam) = f_{N,p,\ga,\al,q}\lt(\lam^{\frac\ga N}\frac{\|\na u\|_p^\ga}{\|u\|_p^\ga}\rt).
\end{equation}
Recall the definition of functions $f_{N,p,\ga,\al,q}$ and $g_{N,p,\ga,q}$ from \eqref{eq:Dformula} and \eqref{eq:alphaformula}. It is clear that $\lim\limits_{t\to 0} f_{N,p,\ga,\al,p}(t) = 1$. As a consequence, we get
\begin{equation}\label{eq:chanduoiD}
D_{N,p,\ga,\al,p} \geq 1.
\end{equation} 

Similar with Lemma \ref{relationcritical}, the same relation between the attainability of $D_{N,p,\ga,\al,q}$ and the attainability of $\sup_{t>0} f_{N,p,\ga,\al,q}(t)$ in $(0,\infty)$ holds. More precisely, we have
\begin{lemma}\label{subrelation}
Let $N \geq 2$, $p \in (1,N]$, $q \in (p,p^*)$ and $\al, \ga >0$. Then $D_{N,p,\ga,\al,q}$ is attained in $S_{N,p,\ga}$ if and only if $\sup_{t>0} f_{N,p,\ga,\al,q}(t)$ is attains in $(0,\infty)$.
\end{lemma}
\begin{proof}
The proof is completely similar with the one of Lemma \ref{relationcritical} by using \eqref{eq:relation} and \eqref{eq:dangthucsub}. Hence we omit it.
\end{proof}

We next prove Theorem \ref{IshiwataWadade}
\begin{proof}[Proof of Theorem \ref{IshiwataWadade}]
We first prove part $(1)$ under the assumption $\ga > \ga_c$. In this case, we have $\lim\limits_{t\to 0} g_{N,p,\ga,q}(t) =0$ which then implies $\al_{N,p,q}(\ga) =0$. We rewrite $f_{N,p,\ga,\al,q}$ as
\[
f_{N,p,\ga,\al,q}(t) = (1+ t)^{-\frac p\ga} + \al B_{N,p,q} t^{\frac {\ga_c}\ga} (1+t)^{-\frac q\ga},
\]
hence
\begin{align*}
f_{N,p,\ga,\al,q}'(t)& = \frac{t^{-\frac{q}\ga -1}}{\ga} \lt(-p(1+t)^{\frac{q-p}\ga} + \al B_{N,p,q}(\ga_c -q)t^{\frac{\ga_c}\ga} + \al B_{N,p,q} \ga_c t^{\frac{\ga_c}\ga -1}\rt)\\
&= \frac{t^{-\frac{q}\ga -1}}{\ga} h_{N,p,\ga,\al,q}(t)
\end{align*}
Since $\ga > \ga_c$ and $\ga_c < q$, then $h_{N,p,\al,\ga,q}$ is strictly decreasing function on $(0,\infty)$. Note that $h_{N,p,\ga,\al,q}(t) \to \infty$ as $t\to 0$ and $h_{N,p,\ga,\al,q}(t) \to -\infty$ as $t\to \infty$. Consequently, there is unique $t_0 \in (0,\infty)$ such that $f_{N,p,\ga,\al,q}'(t_0) =0$, $f_{N,p,\ga,\al,q}'(t) > 0$ if $t \in (0,t_0)$ and $f_{N,p,\ga,\al,q}'(t) < 0$ if $t \in (t_0,\infty)$. Thus $f_{N,p,\ga,\al,q}$ attains its maximum value at unique point $t_0 \in (0,\infty)$. By Lemma \ref{subrelation}, $D_{N,p,\ga,\al,q}$ is attained in $S_{N,p,\ga}$.

We next prove part $(2)$ under the assumption that $\ga = \ga_c$. We first show that $\al_{N,p,q}(\ga_c) =p/(\ga_c B_{N,p,q})$. Indeed, we have $(q-p)/\ga_c = p/N < 1$ and $q/\ga_c >1$, hence the function $h(t) = (1+t)^{q/\ga_c} - (1+t)^{(q-p)/\ga_c}$ is convex on $[0,\infty)$. Consequently, we get $h(t) -h(0) \geq h'(0) t$ which is equivalent to $g_{N,p,\ga,q}(t) \geq p/(\ga_c B_{N,p,q})$ for any $t >0$. It is clear that $\lim\limits_{t\to 0} g_{N,p,\ga,\al} = p/(\ga_c B_{N,p,q})$. These estimates imply $\al_{N,p,q}(\ga_c) =p/(\ga_c B_{N,p,q})$.

As in the proof of part $(1)$, we have
\begin{align*}
f_{N,p,\ga_c,\al,q}'(t) &=\frac{t^{-\frac{q}{\ga_c} -1}}{\ga_c} \lt(-p(1+t)^{\frac{q-p}{\ga_c}} - \al B_{N,p,q}(q-\ga_c)t + \al B_{N,p,q} \ga_c \rt)\\
&=:\frac{t^{-\frac{q}{\ga_c} -1}}{\ga_c} h_{N,p,\ga_c,\al,q}(t).
\end{align*}
Suppose that $\al \leq\alpha_{N,p,q}(\ga_c)$. We then have $f_{N,p,\ga_c,\al,q}'(t) < 0$ for any $t >0$ hence 
\[
f_{N,p,\ga,\al,q}(t) < \lim\limits_{t\to 0} f_{N,p,\ga_c,\al,q}(t) = 1 =\sup_{t> 0} f_{N,p,\ga,\al,q}(t),
\]
for any $t >0$. Hence $\sup_{t> 0} f_{N,p,\ga,\al,q}(t)$ is not attained in $(0,\infty)$. By Lemma \ref{subrelation}, $D_{N,p,\ga_c,\al,q}$ is not attained in $S_{N,p,\ga}$. By \eqref{eq:Dformula}, it is clear that $D_{N,p,\ga,\al,q} =1$.

Suppose that $\al > \al_{N,p,q}(\ga_c)$. We have $\lim_{t\to 0} h_{N,p,\ga_c,\al,q}(t) \to -p + \al B_{N,p,q} \ga_c > 0$ and $\lim_{t\to \infty} h_{N,p,\ga_c,\al,q}(t) = -\infty$ since $\ga_c < q$. Note that $h_{N,p,\ga_c,\al,q}$ is strictly decreasing function on $(0,\infty)$. Hence, there is unique $t_0 \in (0,\infty)$ such that $f_{N,p,\ga_c,\al,q}'(t_0) =0$, $f_{N,p,\ga_c,\al,q}'(t) > 0$ if $t \in (0,t_0)$ and $f_{N,p,\ga_c,\al,q}'(t) < 0$ if $t \in (t_0,\infty)$. Consequently, $\sup_{t>0} f_{N,p,\ga_c,\al,q}(t)$ is attained at the unique point $t_0\in (0,\infty)$. By Lemma \ref{subrelation}, $D_{N,p,\ga_c,\al,q}$ is attained in $S_{N,p,\ga}$.

It remains to prove part $(3)$ under the assumption $\ga < \ga_c$. Suppose that $\al < \al_{N,p,q}(\ga)$. We then have $ \al B_{N,p,q}  < g_{N,p,\ga,\al,q}(t)$ for any $t>0$ which yields $f_{N,p,\ga,\al,q}(t) < 1$ for any $t >0$. This together with \eqref{eq:chanduoiD} implies that 
\[
f_{N,p,\ga,\al,q}(t) < 1 = \sup_{t>0} f_{N,p,\ga,\al,q}(t) =D_{N,p,\ga,\al,q}, 
\]
for any $t>0$. Hence $\sup_{t>0} f_{N,p,\ga,\al,q}(t)$ is not attained in $(0,\infty)$. Hence $D_{N,p,\ga,\al,q}$ is not attained in $S_{N,p,\ga}$ by Lemma \ref{subrelation}.

Suppose that $\al \geq \al_{N,p,q}(\ga)$. Since $\ga < \ga_c < q$, it is clear that 
\[
\lim_{t\to 0} g_{N,p,\ga,q}(t) = \infty = \lim_{t \to \infty} g_{N,p,\ga,q}(t).
\]
Consequently, there exists $t_1 \in (0,\infty)$ such that $ \al_{N,p,q}(\ga) = g_{N,p,\ga,q}(t_1)$ and hence 
\[
\al B_{N,p,q} t_1^{\frac{\ga_c}\ga } \geq \al_{N,p,q}(\ga_c) B_{N,p,q} t_1^{\frac{\ga_c}\ga } =(1+t_1)^{\frac q\ga} - (1+ t_1)^{\frac{q-p}\ga}.
\]
Thus we get
\[
f_{N,p,\ga,\al,q}(t_1) \geq 1 = \lim_{t\to 0} f_{N,p,\ga,\al,q}(t) > 0 = \lim_{t\to \infty} f_{N,p,\ga,\al,q}(t).
\]
Consequently, $\sup_{t>0} f_{N,p,\ga,\al,q}(t)$ is attained in $(0,\infty)$. By Lemma \ref{subrelation}, $D_{N,p,\ga,\al,q}$ is attained in $S_{N,p,\ga}$.
\end{proof}


\section{Further remarks}
In this section, we explain how our method used to prove Theorem \ref{IshiwataWadade} and Theorem \ref{Maximizerscritical} can be applied to the maximizing problems of type \eqref{eq:variationalproblem} for the fractional Laplacian operators. Let $s \in (0,N/2)$, the fractional Laplacian operators $(-\Delta)^s$ is defined via the Fourier transformation by $\widehat{(-\Delta)^s u}(\xi) = |\xi|^{2s} \widehat{u}(\xi)$, where $\widehat{u}(\xi) = \int_{\R^N} e^{-i x\cdot \xi} u(x) dx$ denotes the Fourier transformation of $u$. Let $H^s(\R^N)$ denote the fractional Sobolev space of all functions $u\in L^2(\R^N)$ such that $(-\De)^{s/2} u \in L^2(\R^N)$. For $\gamma >0$, we define
\[
\|u\|_{H^s_\ga} = \lt(\|u\|_2^\ga + \|(-\De)^{\frac s2}u\|_2^\ga\rt)^{\frac1\ga}, \qquad u\in H^s(\R^N).
\]
For $\beta, \ga >0$ and $2 < q \leq 2_s^* := 2N/(N-2s)$, we consider the problem
\begin{equation}\label{eq:fracLaplace}
S_{N,s,\gamma,\beta, q} = \sup_{u\in H^s(\R^N),\, \|u\|_{H^s_\ga} =1} \lt(\|u\|_2^2 + \beta \|u\|_q^q\rt).
\end{equation}
Let $\mathcal S_{N,s,q}$ denote the best constant in the fractional Gagliardo--Nirenberg--Sobolev inequality
\begin{equation}\label{eq:fractionalGNS}
\mathcal S_{N,s,q} = \sup_{u\in H^s(\R^N), \, u\not\equiv 0} \frac{\|u\|_q^q}{\|(-\Delta)^{\frac s2}u\|_2^{\ga_{s,q}} \|u\|_2^{q -\ga_{s,q}}},\qquad \ga_{s,q}: = \frac{N(q-2)}{2s}.
\end{equation}
It is well known that if $q \in (2,2_s^*)$ then $\mathcal S_{N,s,q}$ is attained in $H^s(\R^N)$. If $q = 2_s^*$, \eqref{eq:fractionalGNS} reduces to the fractional Sobolev inequality whose best constant $\mathcal S_{N,s,2_s^*}$ was explicitly computed by Lieb \cite{Lieb}. In fact, Lieb computed the sharp constants in the sharp Hardy--Littlewood--Sobolev inequality which is dual form of the sharp fractional Sobolev inequality. Moreover, the equality holds in the fractional Sobolev inequality if and only if $u(x) =u_{s,*}(x):= (1 + |x|^2)^{(2s-N)/2}$ up to a translation, dilation and multiple by a non-zero constant. Note that $u_{s,*} \not\in H^s(\R^N)$ unless $s < N/4$.

Define
\begin{equation}\label{eq:betafunction}
\beta_{N,s,q}(\ga) = \inf_{u\in H^s(\R^N),\, \|u\|_{H^s_\ga} =1} \frac{1 -\|u\|_2^2}{\|u\|_q^q}.
\end{equation}
Similar to the problem \eqref{eq:variationalproblem}, $\beta_{N,s,q}(\ga)$ will play the role of threshold for the existence of maximizers for the problem \eqref{eq:fracLaplace} (as the one of $\al_{N,p,q}(\ga)$). Following the proof of Theorem \ref{IshiwataWadade} and Theorem \ref{Maximizerscritical}, we can prove the following results for problem \eqref{eq:fracLaplace}.
\begin{theorem}\label{sub}[subcritical case $q < 2_s^*$]
Let $N \geq 1$, $s\in (0, N/2)$, $q\in (2, 2_s^*)$ and $\beta , \ga >0$. Then we have
\begin{description}
\item (1) If $\ga > \ga_{s,q}$ then $\beta_{N,s,q}(\ga) =0$ and $S_{N,s,\gamma,\beta,q}$ is attained for any $\beta >0$.
\item (2) If $\ga =\ga_{s,q}$ then $\beta_{N,s,q}(\ga_{s,q}) = 2/(\ga_{s,q}\, \mathcal S_{N,s,q})$ and $S_{N,s,\ga,\beta,q}$ is attained for any $\beta > \beta_{N,s,q}(\ga_{s,q})$ while it is not attained for any $\beta \leq \beta_{N,s,q}(\ga_{s,q})$.
\item (3) If $\ga < \ga_{s,q}$ then $S_{N,s,\ga,\beta,q}$ is attained for any $\beta \geq \beta_{N,s,q}(\ga)$ while it is not attained for any $\beta < \beta_{N,s,q}(\ga)$.
\end{description}
\end{theorem}
In the critical case $q = 2_s^*$, we have $\ga_{s,2_s^*} = 2_s^*$. In this case, we denote $\mathcal S_{N,s,2_s^*}$, $\beta_{N,s,2_s^*}$ and $S_{N,s,\ga,\beta,2_s^*}$ by $\mathcal S_{N,s}$, $\beta_{N,s}$ and $S_{N,s,\ga,\beta}$ respectively for simplifying notation. We have the following results.
\begin{theorem}\label{critical}[critical case $q = 2_s^*$]
Let $N \geq 1$, $s\in (0, N/2)$ and $\beta, \ga > 0$. If $s \geq N/4$ then $S_{N,s,\ga,\beta}$ is not attained. If $s < N/4$, then the following results hold.
\begin{description}
\item (1) If $\ga > 2_s^*$ then $\beta_{N,s}(\ga) =0$ and $S_{N,s,\ga,\beta}$ is attained for any $\beta >0$.
\item (2) If $\ga =2_s^*$ then $\beta_{N,s}(2_s^*) = 2/(2_s^*\, \mathcal S_{N,s})$ and $S_{N,s,\ga,\beta}$ is attained for any $\beta > \beta_{N,s}(2_s^*)$ while it is not attained for any $\beta \leq \beta_{N,s}(2_s^*)$. Moreover, if $\beta \leq \beta_{N,s}(2_s^*)$ then $S_{N,s,\ga,\beta} =1$.
\item (3) If $2 < \ga < 2_s^*$ then $S_{N,s,\ga,\beta}$ is attained for any $\beta \geq \beta_{N,s}(\ga)$ while it is not attained for any $\beta < \beta_{N,s}(\ga)$. Moreover, if $\beta \leq \beta_{N,s}(\ga)$ then $S_{N,s,\ga,\alpha} =1$. 
\item (4) If $\ga \leq 2$ then $\beta_{N,s}(\ga) =\mathcal S_{N,s}^{-1}$ and $S_{N,s,\ga,\beta}$ is not attained for any $\beta >0$. In this case, we have
\[
S_{N,s,\ga,\beta} = \max\{1, \beta \,\mathcal S_{N,s}\}.
\]
\end{description}
\end{theorem}
The proof of Theorem \ref{sub} and Theorem \ref{critical} follows the same way in the proof of Theorem \ref{IshiwataWadade} and Theorem \ref{Maximizerscritical}. We first can show that
\[
S_{N,s,\ga,\beta,q} = \sup_{t > 0}\, \frac{(1+ t)^{\frac{q-2}\ga} + \beta\, \mathcal S_{N,s,q} t^{\frac{\ga_{s,q}}\ga}}{(1+ t)^{\frac q\ga}} =: \sup_{t> 0}\, \overline f_{N,s,\ga,\beta,q}(t),
\]
and
\[
\beta_{N,s,q}(\ga) = \frac1{\mathcal S_{N,s,q}}\, \inf_{t >0} \, \frac{(1+t)^{\frac{q-2}\ga}((1+t)^{\frac2\ga} -1)}{t^{\frac{\ga_{s,q}}\ga}} = \frac1{\mathcal S_{N,s,q}}\, \inf_{t>0}\, \overline g_{N,s,\ga,q}(t).
\]
The functions $\overline f_{N,s,\ga,\beta,q}$ and $\overline g_{N,s,\ga,q}$ have the same properties of the functions $f_{N,p,\ga,\al,q}$ and $g_{N,p,\ga,q}$ above (also $f_{N,p,\ga,\al}$ and $g_{N,p,\ga}$) respectively. So we can use again the arguments in the proof of Theorem \ref{IshiwataWadade} and Theorem \ref{Maximizerscritical} to prove Theorem \ref{sub} and Theorem \ref{critical}. 

As a consequence, we obtain the existence results for solutions in $H^s(\R^N)$ of several non-local elliptic type equations
\[
\|(-\Delta)^{\frac s2}u\|_2^{\ga -2} (-\Delta)^s u + \lt(\|u\|_2^{\ga -2} -\frac2{2\|u\|_2^2 + \beta q \|u\|_q^q}\rt)u = \frac{\beta q}{2\|u\|_2^2 + \beta q \|u\|_q^q} |u|^{q-2} u
\]
under some suitable conditions on $q \in (2,2_s^*]$ and $\beta, \ga >0$. Note that the previous equation is exact the Euler--Lagrange equation with Lagrange multiplier of the problem \eqref{eq:fracLaplace}.

The similar results for the fractional $p-$Laplace operators $(-\De_p)^s$ where $s \in (0,1)$ and $ps < N$ can be proved by the same way. We refer the reader to the paper \cite{brasco} and references therein for the definition and properties of the operator $(-\De_p)^s$. The fractional Sobolev inequality and fractional Gagliardo--Nirenberg--Sobolev inequality related to $(-\De_p)^s$ are well known \cite{DiNezza,Mazya}. The existence of maximizers for these inequalities can be proved by using the concentration--compactness principle of Lions \cite{PL85I,PL85II}. The optimal estimate for the maximizers of the fractional Sobolev inequality was recently proved in \cite{brasco}. These ingredients are enough to apply our method used to prove Theorem \ref{IshiwataWadade} and Theorem \ref{Maximizerscritical} to establish the similar results for $(-\De_p)^s$. We leave the details for interesting reader.

\section*{Acknowledgments}
This work was supported by the CIMI's postdoctoral research fellowship.

\end{document}